\documentclass{article}
\usepackage{arxiv}

\usepackage[normalem]{ulem} 
\usepackage{todonotes}

\usepackage[utf8]{inputenc} 
\usepackage[T1]{fontenc}    
\usepackage{hyperref}       
\usepackage{url}            
\usepackage{booktabs}       
\usepackage{amsfonts}       
\usepackage{microtype}      
\usepackage{color}
\usepackage{lipsum}
\usepackage{amsthm}
\usepackage{amsmath}
\usepackage{indentfirst}
\usepackage{graphicx}
\usepackage{epstopdf}
\usepackage{fourier}
\usepackage{bm}
\usepackage{bbm} 
\usepackage{mathtools}
\usepackage{latexsym,enumerate}
\usepackage{ulem}

\newcommand{\comment}[1]{}

\newcommand{\BEA}{\begin{eqnarray}}
\newcommand{\EEA}{\end{eqnarray}}

\newcommand{\BR}{\mathbb{R}}

\newtheorem{lem}{Lemma}[section]
\newtheorem{theo}{Theorem}[section]
\newtheorem{prop}{Proposition}[section]
\newtheorem{defi}{Definition}[section]

\newtheorem{remark}{Remark}

\title{A Higher Order Local Mesh Method for Approximating 1-Laplacians on Unknown Manifolds}

\author{
  J. Wilson Peoples \\
  Department of Mathematics \\
  The Pennsylvania State University, University Park, PA 16802, USA\\
  \texttt{peoplesjwilson@gmail.com} \\
  \And
John Harlim \\
  Department of Mathematics, \\ Institute for Computational and Data Sciences \\
  The Pennsylvania State University, University Park, PA 16802, USA\\
  \texttt{jharlim@psu.edu} \\
}

\begin{document}

\maketitle

\begin{abstract}

We introduce a numerical method to approximate arbitrary differential operators on vector fields in weak form, given point cloud data sampled randomly from a $d$-dimensional manifold embedded in $\mathbb{R}^n$. This method generalizes the local linear mesh method to the local curved mesh method, thus allowing for the estimation of differential operators with nontrivial Christoffel symbols, such as the Bochner and Hodge Laplacians. In particular, we leverage the potentially small intrinsic dimension of the manifold $(d \ll n)$ to construct local parameterizations that incorporate both local meshes and higher-order curvature information. The former is constructed using low-dimensional meshes obtained from local data projected onto the tangent spaces, while the latter is obtained by fitting local polynomials using generalized moving least squares. Theoretically, we prove spectral convergence of the proposed method for the estimation of the Bochner Laplacian. We provide numerical results supporting the theoretical convergence rates for the Bochner and Hodge Laplacians on simple manifolds.

\end{abstract}



\section{Introduction}

The graph Laplacian is a widely used tool in unsupervised learning tasks such as dimensionality reduction \cite{belkin2003laplacian,coifman2006diffusion}, clustering \cite{ng2002spectral,von2008consistency}, and semi-supervised learning \cite{belkin2002semi,calder2020properly}. Given a finite sample set $X=\{x_1,\ldots, x_N\}$, the graph Laplacian is constructed by quantifying pairwise affinities between data points. Under the manifold assumption $X\subset M$, where $M$ is a $d$-dimensional Riemannian submanifold of $\mathbb{R}^n$, a standard approach is to construct the graph using a kernel that measures the closeness of data points with respect to the ambient Euclidean metric. The key motivation is that, for sufficiently local kernels \cite{coifman2006diffusion,berry2016local}, including compactly supported kernels \cite{trillos2020error,calder2022improved}, Euclidean distances in $\mathbb{R}^n$ provide accurate approximations of geodesic distances on $M$. This observation firmly establishes the graph Laplacian as a fundamental tool in manifold learning: it enables access to the intrinsic geometry of the data using only ambient-space information. In particular, when the graph Laplacian is shown to converge, in the large-sample limit, to its continuous analog, the Laplace--Beltrami operator, it becomes a consistent estimator of the underlying geometric structure. Since the Laplace--Beltrami operator acts on functions, this class of approaches allows one to represent functions on the manifold underlying the data.

Recent extensions for representing vector fields on manifolds (or, dually, differential 1-forms) have also been proposed. One prominent approach is Vector Diffusion Maps \cite{singer2012vector}, which generalizes Diffusion Maps by incorporating local parallel transport into the kernel construction and has been shown to be consistent for approximating the connection Laplacian. Related frameworks, including Discrete Exterior Calculus (DEC) \cite{hirani2003discrete} and Finite Element Exterior Calculus (FEEC) \cite{arnold2018finite}, have also been developed. However, these methods are often ill-suited for point-cloud data, as they require additional geometric structures, such as global meshes or simplicial complexes, that may not be available in many data-driven settings. In particular, when point-cloud data are randomly sampled from a low-dimensional manifold embedded in a high-dimensional Euclidean space, constructing a global mesh in the high-dimensional ambient space is computationally infeasible. Moreover, the spectral consistency of DEC for estimating Laplace--de Rham operators remains unclear.

These limitations are addressed by the Spectral Exterior Calculus (SEC) \cite{berry2020spectral}, a Galerkin-based framework that represents differential forms and vector fields using eigenfunctions of the Laplace--Beltrami operator. However, the accuracy of SEC depends critically on the quality of the estimated Laplace--Beltrami eigenpairs, since all differential operators are constructed through spectral expansions of the zero-form Laplacian. In \cite{berry2020spectral}, Diffusion Maps are used to estimate these eigenfunctions. We argue that the performance of SEC can be further improved by the approach developed in this paper, which enables accurate estimation of eigensolutions for Laplace operators, including the Laplace--Beltrami, Bochner, and Hodge Laplacians.

In our previous work \cite{harlim2023radial}, we developed a mesh-free framework to estimate arbitrary differential operators acting on tensor fields. The original formulation was based on the Radial Basis Function (RBF) interpolation, and its computational efficiency was later improved by Generalized Moving Least Squares (GMLS) \cite{liang2013solving,jiang2024generalized}. This framework employs moving least-squares to estimate local coordinate charts and subsequently approximates differential operators using standard pointwise formulas from differential geometry. However, because the method is inherently pointwise, it yields a nonsymmetric discretization of the Laplacian operators, which significantly complicates both computation and theoretical analysis. Although a variational formulation of these differential operators was also introduced in \cite{harlim2023radial}, its reliance on Monte-Carlo integration substantially degrades the accuracy of the resulting operator estimates.

In this paper, we build on the local mesh approach of \cite{lai2013local}, replacing Monte-Carlo integration over point cloud data with a higher-order local mesh construction. The original method \cite{liang2013solving} constructs local linear meshes by projecting data points onto estimated tangent spaces inferred from the point cloud. Although this framework enables the estimation of the Laplace–Beltrami operator, it fails to accurately approximate operators whose local expressions depend on higher-order curvature quantities, such as Christoffel symbols. This limitation arises because the induced Riemannian metric is constant on each triangle of the local meshes.

To overcome this issue, we introduce a higher-order local curved mesh based on Generalized Moving Least-Squares (GMLS), generalizing the approach in \cite{lai2013local}. Although this framework can approximate arbitrary differential operators acting on tensor fields, we focus here on estimating 1-Laplacians on vector fields. In particular, we establish spectral consistency with the Bochner Laplacian in the large-data limit. Extensive numerical experiments further verify the convergence of the proposed Hodge and Bochner Laplacian estimators using randomly sampled ambient point-cloud data. Practically, the estimated tensor fields can be directly evaluated at any out-of-sample data points on the manifold, without requiring an additional interpolation scheme, as is needed in graph-Laplacian-based methods, including Vector Diffusion Maps and SEC.

The remainder of this paper is organized as follows. In Section \ref{local-metrics-section}, we review the local frameworks of \cite{lai2013local} and \cite{liang2013solving}. We also demonstrate how the parameterization of \cite{liang2013solving} can be used on the local mesh of \cite{lai2013local} to obtain local curved meshes. In Section \ref{weak-operator-learning}, we describe how the local charts of Section \ref{local-metrics-section} can be used to obtain weak-form approximations of arbitrary differential operators on tensor fields. In Section \ref{theoretical-guarantees-section}, we prove spectral convergence for the estimation of the Bochner Laplacian. Finally, in Section \ref{numerical-results-section}, we demonstrate the numerical convergence of the local curved mesh method for estimating the Hodge and Bochner Laplacians on several example manifolds. We supplement this paper with two appendices. In Appendix~\ref{computation-appendix}, we report detailed calculations of the components of the mass and stiffness matrices for both the Bochner and Hodge Laplacians. In Appendix~\ref{proofs-of-technical-lemmas}, we provide proofs of some technical lemmas.

\section{Local Metric Learning}
\label{local-metrics-section}
This section provides an overview of two existing frameworks that are key to the novel method introduced in this paper. We begin by describing the Local Mesh Method \cite{lai2013local}, which leverages local tangent space approximations to construct triangular meshes near each point of a given point cloud. We also present a brief overview of the Generalized Moving Least-Squares technique \cite{liang2013solving}, in which local parameterizations of the form $\Phi: T_{x}M \to M$ are approximated by fitting polynomials on the estimated tangent spaces using weighted least-squares \textcolor{black}{approach}. Subsequently, we introduce the Local Curved Mesh Method, which serves as the central idea of this paper, \textcolor{black}{and which amalgamates} the aforementioned methodologies. To allow for a concise description of each method, we focus on learning the Riemannian metric of an unknown manifold on a local chart throughout this section. The ideas presented in this section, however, can be extended to approximate more complicated objects on \textcolor{black}{a} manifold (e.g., operators on tensor bundles, as shown in Section \ref{weak-operator-learning}).

\subsection{Local tangent space approximation}
\label{local-tangent-space-subsection}

Given point cloud data $X = \{x_1, \dots, x_N\} \subseteq \mathbb{R}^n$ sampled from an unknown manifold $M$ of dimension $d \ll n$, we wish to obtain {\color{black}\color{black}a local estimate of} $T_{x_i}M$, the tangent space at $x_i$. This problem has been well studied in the literature and is frequently addressed using local SVD \cite{donoho2003hessian,tyagi2013tangent,zhang2004principal} or local PCA, as in \cite{lai2013local,liang2013solving}. More recently, a second-order local SVD method introduced in \cite{harlim2023radial} improves upon local SVD by keeping track of higher-order terms. For simplicity, we provide a brief overview of local PCA.

Given a point $x_i$, denote its $k$-nearest neighbors by $K(i) = \{x_{i_1}, \dots, x_{i_k}\}$. In this notation, $x_i=x_{i_1}$. Consider the local $n\times n$ covariance matrix
$$
P_i = \sum_{x \in K(i)} (x - \mu_i)(x - \mu_i)^\top,
$$
where $\mu_i = \frac{1}{k} \sum_{x \in K(i)} x$. In practice, we found insignificant differences in accuracy when using either the average $\mu_i$ or the base point $x_i$ {\color{black}as the reference point in place of $\mu_i$}, away from the boundary. After diagonalizing this matrix, one obtains $n$ eigenpairs $\left(\lambda^{(i)}_1, \mathbf{t}^{(i)}_1\right), \dots, \left(\lambda^{(i)}_n, \mathbf{t}^{(i)}_n\right)$ with $\lambda^{(i)}_j \geq \lambda^{(i)}_{j+1}$. Since the intrinsic dimension of the data is $d$, one expects
$$
\lambda^{(i)}_1 \geq \lambda^{(i)}_2 \geq \dots \geq \lambda^{(i)}_d \gg \lambda^{(i)}_{d+1},
$$
provided that the sampling is sufficiently dense around $x_i$. Thus, one can determine the dimension $d$ and obtain an orthonormal basis $\left\{\mathbf{t}^{(i)}_1, \dots, \mathbf{t}^{(i)}_d\right\}$ that approximates the tangent space at $x_i$. We denote the copy of $\mathbb{R}^d$ given by the span of these {\color{black}basis vectors} by $\overline{T_{x_i} M}$. In {\color{black}the remainder of this manuscript}, we assume that such an estimate is {\color{black}available.}

\subsection{Approximation of metric tensor with linear local meshes}
\label{local-mesh-section}
Given $x_i$, we aim to obtain a $d$-dimensional mesh for $K(i)$, which is $x_i$ and its $k$-nearest neighbors. To simplify the discussion, we assume $d=2$ {\color{black}and $n=3$}. In \cite{lai2013local}, {\color{black}the authors} propose to solve this problem by first projecting $K(i)$ {\color{black}onto} the estimated tangent space:
\BEA
T(i) := \left\{ \underbrace{\left(\mathbf{t}^{(i)}_1 \cdot (x_{i_1} - x_{i_1}), \mathbf{t}^{(i)}_2 \cdot (x_{i_1} - x_{i_1})\right)}_{(0,0)= \vec{0}}, \dots, \left(\mathbf{t}^{(i)}_1 \cdot (x_{i_k} - x_{i_1}), \mathbf{t}^{(i)}_2 \cdot (x_{i_k} - x_{i_1}) \right) \right\} \subseteq \overline{T_{x_i} M}. \notag
\EEA
For simplicity of future notation, we define 
\BEA
\vec{v}_{i_r} = 
\begin{pmatrix} 
\mathbf{t}^{(i)}_1 \cdot (x_{i_r} - x_{i_1}) \\ 
\mathbf{t}^{(i)}_2 \cdot (x_{i_r} - x_{i_1})
\end{pmatrix}, \label{approximate-tangent-data-equation}
\EEA
so that 
\BEA
T(i) = \left\{ \vec{v}_{i_1}=\vec{0}, \vec{v}_{i_2}, \dots, \vec{v}_{i_k} \right\}. \notag
\EEA
A mesh on $T(i)$, denoted by $\mathbf{R}(i)$, is obtained using standard methods, e.g., the {\color{black}Delaunay} triangulation. Denote by $R(i)$ the first ring of this triangulation, that is, $R(i)$ is {\color{black}the} set of $\ell$ triangles of this mesh, which includes the base point $(0,0)$. These necessarily take the form
$$
R(i) =  \left\{ T_{i,1}, \dots , T_{i,\ell}\right\} \subset \mathbf{R}(i),
$$
where each $T_{i,r} = [\vec{0}, \vec{v}_{i_{j_r}}, \vec{v}_{i_{k_r}}]$. Consider any such triangle. This can be parameterized by the standard Barycentric parameterization with the canonical right triangle $T = \{ (u_1, u_2) | u_1, u_2 \geq 0, u_1 + u_2 \leq 1 \}$  via the map $\Phi_{T_{i,r}}: T \to T_{i,r} \subseteq \mathbb{R}^2$ defined by 
\BEA
\label{triangle-to-intrinsic-map}
(u_1, u_2) \mapsto u_1 \vec{v}_{i_{j_r}} + u_2 \vec{v}_{i_{k_r}}. 
\EEA
Since the first ring can be ``lifted" from the tangent space to the manifold via the one-to-one correspondence between $T(i)$ and $K(i)$ given by 
$$
\vec{v}_{i_r} \mapsto x_{i_r}.
$$
we obtain the map to the global coordinates, $\tilde{\Phi}_{T_{i,r}}: T \to \mathbb{R}^3$, defined by 
\BEA
\label{triangle-to-ambient-map}
(u_1, u_2) \mapsto u_1 (x_{i_{j_r}} - x_i) + u_2 (x_{i_{k_r}} - x_i) + x_i. 
\EEA
Let $y\in \tilde{\Phi}_{T_{i,r}}(T)\subset \BR^{\color{black}3}$. Based on the barycentric parameterization, there exists $(u_1,u_2)\in T$ such that
\[
y = (1-u_1-u_2)x_i + u_1 x_{i_{j_r}} + u_2 x_{i_{k_r}} = x_i + u_1 (x_i - x_{i_{j_r}}) + u_2 (x_i - x_{i_{k_r}}) := \tilde{\Phi}_{T_{i,r}}(u_1,u_2),
\]
which is effectively Eq (3) above. From this identity, it is clear that
\[
\underbrace{[{\bf t}^{(1)}_1 {\bf t}^{(1)}_2]^\top}_{={\bf T}^\top} (y-x_i) = u_1 \vec{v}_{i_{j_r}}+u_2 \vec{v}_{i_{k_r}} = \Phi_{T_{i,r}}(u_1,u_2).
\]
Then
\BEA
y - x_i &=& {\bf T}{\bf T}^\top(y-x_i) + {{\bf t}^{(i)}_3}{{\bf t}^{(i)}_3}^\top(y-x_i) \notag,
\EEA
so we can obtain the relation between $\tilde{\Phi}$ in terms of $\Phi$,
\BEA 
\tilde{\Phi}_{T_{i,r}}(u_1,u_2) &=& {\bf T} \Phi_{T_{i,r}}(u_1,u_2) + {{\bf t}^{(i)}_3}{{\bf t}^{(i)}_3}^\top(y-x_i)  +x_i. \label{triangle-to-ambient-map2}
\EEA
In this case, the 'lifted' triangle can be understood as a linear interpolation in the normal component given by
\[
p_{i,r} (u_1,u_2) = {{\bf t}^{(i)}_3}^\top(y-x_i)  = u_1 {{\bf t}^{(i)}_3}^\top(x_i-x_{i_{j_r}}) + u_2 {{\bf t}^{(i)}_3}^\top(x_i-x_{i_{k_r}}), 
\]
where $p_{i,r}: T \to \BR$ is a linear function.

The coordinates corresponding to the map $\tilde{\Phi}_{T_{i,r}}$ {\color{black}are} then given by
$$
\left\{ \partial_{u_1} = x_{i_{j_r}} - x_i, \partial_{u_2} = x_{i_{k_r}} - x_i \right\},
$$ 
in which the Riemannian metric takes the form
\BEA
g_{\tilde{\Phi}_{T_{i,r}}}(u_1, u_2) = \begin{bmatrix} 
\|x_{i_{j_r}} - x_i\|^2 & (x_{i_{j_r}} - x_i) \cdot (x_{i_{k_r}} - x_i) \\
(x_{i_{k_r}} - x_i) \cdot (x_{j_{k_r}} - x_i) &  \| x_{i_{k_r}} - x_i \|^2 
\end{bmatrix}. \label{lmm-metric-equation}
\EEA
We emphasize that in this formulation the Riemannian metric is constant as a matrix-valued function of $(u_1, u_2).$ The work in \cite{lai2013local} leverages these coordinate representations, along with local coordinate invariant formulas, to approximate other geometric objects, such as the Laplace--Beltrami operator. 

\begin{remark}
\label{local-mesh-issues-remark}
Since the above metric formula is constant in $(u_1,u_2)$, it is clear that {\color{black}an operator defined} on the above local mesh parameterization cannot distinguish between operators that differ by higher-order geometric objects (e.g., the Hodge and Bochner Laplacians {\color{black}differ by} the Ricci curvature). This issue, which arises from the trivial {\color{black}Christoffel} symbols associated with this linear mesh parameterization, motivates the use of curved (non-flat) local meshes, as discussed in Section \ref{curved-mesh-metric}. In Section \ref{weak-operator-learning}, we will demonstrate how such a general curved local mesh framework can be used to learn operators on general tensor fields.
\end{remark}

\subsection{Generalized moving least-squares approximation}
\label{moving-least-squares-section}
As in {\color{black}the} previous section, we set $d=2$ and $n=3$ to simplify the discussion. Details for extending these ideas to the more general setting can be found in \cite{liang2013solving}. Given a point $x_i$, consider the basis $\big\{ \mathbf{t}^{(i)}_1, \mathbf{t}^{(i)}_2, \mathbf{t}^{(i)}_3 \big\}$, where $\big\{\mathbf{t}^{(i)}_1, \mathbf{t}^{(i)}_2\big\}$ is a basis for the approximated tangent space ${\color{black}\overline{T_{x_i}M}}$, and $\mathbf{t}^{(i)}_3$ approximates the normal direction. {\color{black}Recall that for the set of $k-$nearest neighbors $K(i)$, we denote the set of tangential components of the nearest neighbors by $T(i)$.} Now, let us denote by $N(i)$ the normal components:
$$
N(i) := \left\{ \mathbf{t}^{(i)}_3 \cdot (x_{i_1} - x_{i_1}) , \dots, \mathbf{t}^{(i)}_3 \cdot (x_{i_k} - x_{i_1}) \right\}.
$$
Consider a polynomial $p_i: \overline{T_{x_i} M} \to \mathbb{R}$ defined by 
\BEA
\label{local-polynomial}
p_i(v_1,v_2) = av_1^2 + bv_2^2 + cv_1v_2 + dv_1 + ev_2 + f,
\EEA 
where the coefficients are obtained by fitting the weighted least-squares to the data in $T(i)$ with labels given by $N(i)$. In particular, the coefficients are obtained via 
\BEA
(a,b,c,d,e,f) = \textup{argmin} \sum_{j=1}^k \omega(\| x_i - x_{i_j} \|) \left(p_i( \vec{v}_{i_j} ) -  \mathbf{t}^{(i)}_3 \cdot (x_{i_j}-x_{i_1})\right)^2 \notag
\EEA
for some weight function $\omega$, where $\vec{v}_{i_j}$ is as defined in Equation \eqref{approximate-tangent-data-equation}. Following \cite{liang2012geometric,liang2013solving}, we will consider only the weight function
$$
\omega(t) = \begin{cases} 1 \qquad t=0 \\
1/k \qquad \textup{else}, 
\end{cases}
$$
which is empirically found to be adequate for random data. {\color{black}We will discuss the choice of $k$ in Section~\ref{metric-error-quantification-subsection} to ensure the injectivity of the local projection.} After fitting such a polynomial, we obtain a local parameterization of the form 
\BEA 
\alpha_{x_i}: (v_1, v_2) \mapsto (v_1,v_2, p_i(v_1,v_2)).
\label{mls-parameterization-equation}
\EEA
This parameterization induces the following coordinates: 
$$
\left\{ \partial_{v_1} = 
\begin{bmatrix}
    1 \\
    0 \\
    2av_1 + cv_2 + d
\end{bmatrix} 
\partial_{v_2} = 
\begin{bmatrix}
    0 \\
    1 \\
    2bv_2 + cv_1 + e
\end{bmatrix}
\right\}. 
$$ 
{\color{black}In these coordinates,} the Riemannian metric takes the form
$$
g_{x_i}(v_1, v_2) = \begin{bmatrix}
1 + (2av_1 + cv_2 + d)^2 & (2av_1 + cv_2 + d)(2bv_2 + cv_1 + e) \\
(2av_1 + cv_2 + d)(2bv_2 + cv_1 + e) & 1 + (2bv_2 + cv_1 + e)^2
\end{bmatrix}.
$$
We remark that while the domain of the parameterization from the previous subsection is a standard simplex, the domain of the parameterization from this subsection is the estimated tangent space itself.

\subsection{Approximation of metric tensor with local curved meshes}
\label{curved-mesh-metric}
The local curved meshes conceptually combine the quadratic parameterization of the previous section with the local meshes of Section \ref{local-mesh-section}. As we shall see in Section \ref{weak-operator-learning}, this approximation allows {\color{black}for} flexibility and robustness in using random data in the weak approximation of differential operators while simultaneously avoiding the issues mentioned in Remark \ref{local-mesh-issues-remark}. To this end, consider a fixed point $x_i$ and {\color{black}its} first ring $R(i)$. For an arbitrary triangle 
$$
T_{i,r} = [0, \vec{v}_{i_{j_r}}, \vec{v}_{i_{k_r}}] \in R(i),
$$
recall the map $\Phi_{T_{i,r}}$ as defined in Equation \ref{triangle-to-intrinsic-map}. The {\color{black}range of this map is $T_{i,r}\subset \overline{T_{x_i}M}$, which is a proper subset of} the domain of the parameterization $p_i$ from the previous section. Hence, it is natural to consider parameterizations $\tilde{\Phi}_{C_{i,r}}: T \to \mathbb{R}^3$ of the curved triangular patches of the locally approximated manifold, given by
\BEA
\tilde{\Phi}_{C_{i,r}}(u_1,u_2) &=& {\bf T} \Phi_{T_{i,r}}(u_1,u_2) + {{\bf t}^{(i)}_3}p_{i,r}(u_1,u_2)  +x_i, \label{PhitildeC}
\EEA
 where we generalize \eqref{triangle-to-ambient-map2} by setting the vertical component $p_{i,r}: T\to \BR$ to be a nonlinear map
 \[
 p_{i,r}(u_1,u_2) = p_i \circ \Phi_{T_{i,r}}(u_1,u_2),
 \]
 where $p_i$ is obtained by the generalized moving least-squares procedure discussed in Section~\ref{moving-least-squares-section}. The map $\tilde{\Phi}_{C_{i,r}}$ is effectively a generalization of $\tilde{\Phi}_{T_{i,r}}$, where we embed a curved (non-flat) mesh in $\BR^3$ using a nonlinear map $p_{i,r}$. 
 
Based on this map, one can construct the Riemannian metric $g_{\tilde{\Phi}_{C_{i,r}}}$ corresponding to the curved mesh, generalizing \eqref{lmm-metric-equation}. Since the computation will be performed locally on each {\color{black}base point} $x_i$, it is more convenient to implement the construction without mapping to the global coordinates {\color{black}as follows}
\BEA
\Phi_{C_{i,r}}(u_1, u_2) :=  \begin{pmatrix} 
\Phi_{T_{i,r}} (u_1, u_2) \\
p_i \circ \Phi_{T_{i,r}} (u_1,u_2)\notag
\end{pmatrix}. 
\EEA
This results in coordinates 
\BEA 
\left\{ 
\partial_{u_1} = 
\begin{bmatrix} 
\vec{v}_{i_{j_r}} \\
 \nabla p_i (\Phi_{T_{i,r}}(u_1,u_2)) \cdot \vec{v}_{i_{j_r}}
\end{bmatrix}, 
\partial_{u_2} = 
\begin{bmatrix} 
\vec{v}_{i_{k_r}} \\
 \nabla p_i (\Phi_{T_{i,r}}(u_1,u_2)) \cdot \vec{v}_{i_{k_r}}
\end{bmatrix}
\label{cmm-basis-equation}
\right\}
\EEA 
In these coordinates, the Riemannian metric takes the form 
\BEA
g_{\Phi_{C_{i,r}}}(u_1,u_2) = 
\begin{bmatrix} 
\| \vec{v}_{i_{j_r}} \|^ 2 & \vec{v}_{i_{j_r}} \cdot \vec{v}_{i_{k_r}} \\
\vec{v}_{i_{j_r}} \cdot \vec{v}_{i_{k_r}} & \| \vec{v}_{i_{k_r}} \|^2
\end{bmatrix} + 
\begin{bmatrix} 
 \left( \nabla p_i (\Phi_{T_{i,r}}(u_1,u_2)) \cdot \vec{v}_{i_{j_r}} \right)^2 &  c^i_{j_rk_r}(u_1,u_2) \\
c^i_{j_rk_r}(u_1,u_2) &  \left( \nabla p_i (\Phi_{T_{i,r}}(u_1,u_2))  \cdot \vec{v}_{i_{k_r}} \right)^2 
\end{bmatrix} 
\label{cmm-metric-equation}
\EEA
where $c^i_{j_rk_r}(u_1,u_2) = \left( \nabla p_i (\Phi_{T_{i,r}}(u_1,u_2)) \cdot \vec{v}_{i_{j_r}} \right) \cdot\left( \nabla p_i (\Phi_{T_{i,r}}(u_1,u_2)) \cdot \vec{v}_{i_{k_r}} \right)$.

\section{Operator Approximation in Weak-Form}
\label{weak-operator-learning}
The local mesh method of \cite{lai2013local}, and hence the curved mesh method described in Section \ref{curved-mesh-metric}, lends itself naturally to weak operator approximations, analogous to finite element frameworks. Although the work in \cite{lai2013local} outlines such an approach for approximating the Laplace--Beltrami operator on functions, the same framework can be generalized to differential operators on tensor fields. For instance, this framework can be used to estimate Laplacians on vector fields. In this section, we outline this process. We begin by describing how the Laplace--Beltrami operator can be estimated, highlighting the key differences between the approach in \cite{lai2013local} and the curved mesh method. We then generalize this framework, in the case of the curved mesh method, to arbitrary operators on tensor fields. Since operators on vector fields, such as the Bochner and Hodge Laplacians, are of particular interest, we provide additional details for these cases.

\subsection{Weak approximation of the Laplace--Beltrami operator}\label{WeakLB}

The weak-form Laplace--Beltrami eigenvalue problem is to find eigenpairs $\lambda$ and $f\in H^1(M)$ that {\color{black}satisfy} 
\BEA
\int_M \langle \nabla f, \nabla h \rangle_g d\textup{Vol}= \lambda \int_M f h d\textup{Vol},\label{LBprob}
\EEA
for all $h\in H^1(M)$.
Following the standard finite-element approach, consider a set of ``hat'' functions $$V_N = \left\{f = \sum_{i=1}^N f(x_i) e_i,\,\, e_i : M \to \mathbb{R} \right\} \subset H^1(M),$$ where $\{e_i\}$ are nodal basis {\color{black}functions} satisfying the property $e_i(x_j) = \delta_{ij}.$ Specifically, if the true local parameterization $\gamma_{x_i}: \mathbf{R}(i) \to \mathbb{R}^3$ is denoted by
\BEA 
\gamma_{x_i}: (v_1, v_2) \mapsto x_i + (v_1, v_2, z_{x_i}(v_1,v_2)), \notag 
\EEA
using the canonical map $\Phi_{T_{i,r}}: T \to T_{i,r} $ as defined in \eqref{triangle-to-intrinsic-map}, one can define the nodal basis as a hat function given as
\[
e_i \circ \gamma_{x_i} \circ \Phi_{T_{i,r}} (u_1,u_2) = 1-u_1-u_2. 
\]
Define a parameterization $\Upsilon_{i,r} := \gamma_{x_i} \circ \Phi_{T_{i,r}}: T \to \mathbb{R}^3$. For each pair $\{x_i, x_j\}$ consider the quantity 
\BEA
\label{function-stiffness-matrix-equation}
\int_M \langle \nabla e_i, \nabla e_j \rangle_g d\text{Vol} = \sum_{ T_{i,r} \in R(i) } \int_{T} \langle \nabla_{\Upsilon_{i,r}} e_i , \nabla_{\Upsilon_{i,r}} e_j \rangle_{g_{\Upsilon_{i,r}}} \sqrt{\textup{det}(g_{\Upsilon_{i,r}})} du_1 du_2 := \text{S}_{ij},
\EEA
where $\nabla_{\Upsilon_{i,r}}$ and $g_{\Upsilon_{i,r}}$ denote the gradient and Riemannian metric, respectively, each calculated in the coordinates of the parameterization $\Upsilon_{i,r}: T \to \mathbb{R}^3$. 

When the manifold is unknown as in our case, $\gamma_{i,r}$ is not available. If we consider approximating the linear mesh with parameterization $\tilde{\Phi}_{T_{i,r}}$ as defined in \eqref{triangle-to-ambient-map}, then the hat function $e_i$ is not defined exactly on $M$, but on the linear triangulation $\tilde{\Phi}_{T_{i,r}}(T)$ that locally approximates $M$. In such a case, the function space $V_N \approx \widehat{V}_N = \{f = \sum_{i=1}^N f(x_i) e_i, e_i : \tilde{\Phi}_{T_{i,r}}(T) \to \mathbb{R}\}$. Then
\[
S_{ij} \approx \widehat{S}_{ij}: = \sum_{ T_{i,r} \in R(i) } \int_{T} \langle \nabla_{\tilde{\Phi}_{T_{i,r}}} e_i , \nabla_{\tilde{\Phi}_{T_{i,r}}} e_j \rangle_{g_{\tilde{\Phi}_{T_{i,r}}}} \sqrt{\textup{det}(g_{\tilde{\Phi}_{T_{i,r}}})} du_1 du_2.
\]
The only nonzero terms in the above sum are when $i=j$ or when $\overline{x_i x_j}$ is a side in the lifted triangulation of $T(i)$. Such integrals can be computed analytically. In fact, when hat functions are used, {\color{black}this  formulation yields} the famous cotangent formula for the Laplace--Beltrami operator \cite{duffin1959distributed}. Denote by $\widehat{\mathbf{S}}$ the $N \times N$ sparse matrix with $(i,j)$-th entry $(\mathbf{\widehat{S}})_{ij} = {\widehat{S}}_{ij}$. Similarly, we can approximate  
\BEA
\label{function-mass-matrix-equation}
\int_M e_i e_j d\textup{Vol} \approx \widehat{M}_{ij} := \sum_{T_{i,r} \in R(i)} \int_{T} e_i(u_1,u_2) e_j(u_1,u_2) \sqrt{\det(g_{\tilde{\Phi}_{T_{i,r}}}) } du_1 du_2,  \label{massmatrix}
\EEA
where the approximation occurs when we choose $e_i \in \widehat{V}_N$ in $\widehat{M}_{ij}$.

By definition, $\widehat{\mathbf{S}}$ and $\widehat{\mathbf{M}}$ are not necessarily symmetric.  But since the left-hand {\color{black}terms} in \eqref{function-stiffness-matrix-equation} and \eqref{massmatrix} are symmetric assuming that we have access to the underlying metric $g$, we can approximate the stiffness and mass matrices by
$
\mathbf{\widehat{A}} = \frac{1}{2}(\mathbf{\widehat{S}} + \mathbf{\widehat{S}}^\top),
$ 
and $\mathbf{\widehat{B}} = \frac{1}{2}(\mathbf{\widehat{M}} + \mathbf{\widehat{M}}^\top)$, following \cite{lai2013local}. For convenience of discussion, we also define bilinear operators corresponding to these matrices
\[
a(e_i,e_j) = (\mathbf{\widehat{A}})_{ij}, \quad b(e_i,e_j) =  (\mathbf{\widehat{B}})_{ij}.
\]

Numerically, we approximate the weak-form Laplace--Beltrami eigenvalue problem in \eqref{LBprob}, by finding  $\lambda$ and $f \in \widehat{V}_N$, such that
\[
a(f,h)  = \lambda b(f,h),   
\]
for any $h\in \widehat{V}_N$, which is analogous to solving the following eigenvalue problem:
$$
\left\langle \mathbf{\widehat{A}} \mathbf{f} , \mathbf{h} \right\rangle = \lambda \left\langle \mathbf{\widehat{B}} \mathbf{f}, \mathbf{h} \right\rangle,
$$
where $\mathbf{f} = ({f}_1, \dots, {f}_N)^\top, \mathbf{h} = ({h}_1, \dots, {h}_N)^\top$,  and $\langle \cdot,\cdot \rangle$ denotes the Euclidean dot product (which will be used accordingly for the remainder of this paper). The eigenpair  ($\lambda,\mathbf{f}$) that solves the generalized eigenvalue problem $\mathbf{\widehat{A}} \mathbf{f} = {\lambda} \mathbf{\widehat{B}} \mathbf{f}$ is an estimate {\color{black}of} the eigensolutions of the Laplace--Beltrami operator. Once $\mathbf{f}$ is estimated, the corresponding estimated eigenfunction is given by ${f} = \sum_{i=1}^N {f}_ie_i \in \widehat{V}_N$ that can be used on new (out-of-sample) data points.  

\begin{remark}
In \cite{lai2013local}, the diagonal elements of ${\bf \widehat{S}}$ are chosen in terms of the off-diagonals to ensure that the resulting problem has non-negative eigenvalues, as well as a trivial eigenvalue with {\color{black}a} corresponding eigenvector {\color{black}that is } constant. Since we aim to estimate operators that do not necessarily have such properties, we did not employ such an adjustment in our formulation and implementation. Moreover, explicit simplifications of the integrals \eqref{function-mass-matrix-equation} and \eqref{function-stiffness-matrix-equation} can be found in \cite{lai2013local} in the case of flat local meshes with parameterization $\tilde{\Phi}_{T_{i,r}}$ in \eqref{triangle-to-ambient-map}. We approximate the integral in the stiffness and mass matrices, $\bf{\widehat{S}}$ and $\bf{\widehat{M}}$, {\color{black}using} a first-order quadrature based on evaluating the nodes of the triangle $T$.
\end{remark}

\subsection{Generalization to operators on arbitrary tensor fields} 
\label{generation-subsection}

In this section, we will generalize the framework from the previous subsection to find the eigenvalue $\lambda$ and an eigenvector field $W \in H^1(TM)$ that {\color{black}satisfy}
\BEA
\left \langle DW, DV \right\rangle_{U} = \lambda \langle W, V \rangle_{T^{(1,0)}M},\label{Bocheigprob}
\EEA 
for all $V \in H^1(TM)$. Here, the left-hand inner product is defined for an appropriate range space of $D: H^1(TM) \to U$. For the Bochner Laplacian, $D:=\text{grad}_g:  H^1(TM) \to U = T^{(2,0)}M$ is the gradient of the vector field. For Hodge Laplacian, $D: H^1(TM) \to U = T^{(2,0)}M \bigoplus C^{\infty}(M)$ defined as $D=(\nabla_g, d^\star\flat)$, where the first component is the gradient of vector field and $d^\star$ in the second component denotes the codifferential of the exterior derivative, $d$, acting to 1-form, defined as, $d^\star = (-1)^{d-1} \text{sign}(g) \star d \star$, where $\star:\Omega^k(M) \to \Omega^{d-k}(M)$ is the standard Hodge star operator,  
and the musical notation $\flat$ denotes the ``lowering'' operation that takes {\color{black}a} vector field to its {\color{black}corresponding} covector field (1-form in our case).
 
For clarity, we present the case of $d=2$ and $n=3$. However, this setup can be easily extended to estimate operators on arbitrary tensor fields over manifolds of higher dimension. First, we introduce a finite element space. A natural choice is to consider 
\BEA
\mathcal{W}_N = \left\{ W = \sum_{i=1}^N\sum_{\ell=1}^d W^{(i)}_\ell e_i \mathbf{t}^{(i)}_\ell,\, e_i: M \to \BR\right\} \subset H^1(TM), \label{WN}
\EEA
{\color{black}which is analogous to} $V_N$ in previous section, where the orthonormal basis $\{\mathbf{t}_{\ell}^{(i)}\}$ of local tangent space at $x_i$ are defined in Section~\ref{local-tangent-space-subsection}. Here, the coefficients are given as ${W}_\ell^{(i)} = (\mathbf{t}^{(i)}_\ell)^\top {W}(x_i)$. 

To facilitate the analysis, we consider an intermediate step, where we approximate the left- and right-hand sides of \eqref{Bocheigprob} with symmetric bilinear forms, respectively 
\BEA
a(W_i,W_j) &=& \frac{1}{2}\left( \left \langle D W_i, D W_j \right\rangle_{U} +  \left \langle DW_j, D W_i \right\rangle_{U} \right),\notag \\
b(W_i,W_j) &=& \frac{1}{2}\left( \left \langle W_i, W_j \right\rangle_{T^{(1,0)}M} +   \left \langle W_j, W_i \right\rangle_{T^{(1,0)}M} \right),   \notag
\EEA
for $W_1,W_j \in \mathcal{W}_N$, where each of the inner products in the right-hand terms 
may not be symmetric when they are realized by a local parameterization, $\Upsilon_{i,r}$, which is also used in \eqref{function-stiffness-matrix-equation}. Theoretically, for any $V,W\in H^1(TM)$ and the underlying Riemannian metric $g$ that is unknown in our case, $a(W,V) = \langle DW, DV \rangle_U$ and $b(W,V) = \langle W,V \rangle_{T^{(1,0)}M}$. 

Our goal is to approximate \eqref{Bocheigprob} by finding $\lambda$ and $W \in \mathcal{W}_N$ that satisfy
\BEA
a(W,V) = \lambda b(W,V), \label{variationalBochnereigvalprob}
\EEA
for all $V\in \mathcal{W}_N$, which is equivalent to solving the generalized eigenvalue problem
\BEA
\mathbf{A} \mathbf{W} = \lambda \mathbf{B} \mathbf{W},\label{gen-eigvalprob}
\EEA
where $\lambda$ denotes the eigenvalue corresponding to the eigenvector field $\mathbf{W} \in \BR^{2N}$ with entries given by 
\BEA 
\begin{pmatrix}
   (\mathbf{W})_{2i+1} \\
    (\mathbf{W})_{2i+2}
\end{pmatrix} =  {\color{black} \begin{pmatrix}
    (\mathbf{t}^{(i)}_1)^\top W(x_i) \\
     (\mathbf{t}^{(i)}_2)^\top W(x_i)
\end{pmatrix}} = \begin{pmatrix}W_1^{(i)} \\ W_2^{(i)}\end{pmatrix},
\label{vecW}
\EEA
and $\mathbf{A} = \frac{1}{2}(\mathbf{S}+ \mathbf{S}^\top)$ and $\mathbf{B} = \frac{1}{2}(\mathbf{M}+ \mathbf{M}^\top)$ are the symmetrization of the stiffness and mass matrices $\mathbf{S}$ and $\mathbf{M}$, respectively. 

In this case, the mass and stiffness matrices are sparse $2N \times 2N$ matrices. It is convenient to define these matrices in terms of their $2 \times 2$ block submatrices. In particular, the mass matrix $\mathbf{M}$ has $2\times 2$ block entries given by 
\BEA
\begin{bmatrix} [\mathbf{M}]_{2i-1, 2j -1} & [\mathbf{M}]_{2i-1,2j} \\ [\mathbf{M}]_{2i, 2j-1} & [\mathbf{M}]_{2i,2j} \end{bmatrix} = \sum_{T_{i,r} \in R(i)} \int_{T}  e_i e_j
\begin{bmatrix} 
 \langle \mathbf{t}^{(i)}_{1}, \mathbf{t}^{(j)}_{1}   \rangle_{g_{\Upsilon_{i,r}}}  &  \langle \mathbf{t}^{(i)}_{1}, \mathbf{t}^{(j)}_{2}   \rangle_{g_{\Upsilon_{i,r}}} \\
 \langle \mathbf{t}^{(i)}_{2}, \mathbf{t}^{(j)}_{1}   \rangle_{g_{\Upsilon_{i,r}}} &  \langle \mathbf{t}^{(i)}_{2}, \mathbf{t}^{(j)}_{2}   \rangle_{g_{\Upsilon_{i,r}}}
\end{bmatrix} \sqrt{\det({g_{\Upsilon_{i,r}}})} du_1 du_2,\label{Massmatrix}
\EEA
where the integrals above are taken entrywise, and the inner-products are Riemannian inner-products between vector fields. The stiffness matrix $\mathbf{S}$ has entries 
\BEA
\begin{bmatrix} [\mathbf{S}]_{2i-1, 2j -1} & [\mathbf{S}]_{2i-1,2j} \\ [\mathbf{S}]_{2i, 2j-1} & [\mathbf{S}]_{2i,2j} \end{bmatrix} =
\sum_{T_{i,r} \in R(i)} \displaystyle\int_{T}  
\begin{bmatrix} 
 \langle D e_i \mathbf{t}^{(i)}_{1}, D e_j \mathbf{t}^{(j)}_{1}   \rangle_{g_{\Upsilon_{i,r}}}  &  \langle D e_i \mathbf{t}^{(i)}_{1}, 
 D e_j \mathbf{t}^{(j)}_{2}   \rangle_{g_{\Upsilon_{i,r}}} \\
 \langle  D e_i \mathbf{t}^{(i)}_{2}, D e_j \mathbf{t}^{(j)}_{1}   \rangle_{g_{\Upsilon_{i,r}}} &  \langle D e_i \mathbf{t}^{(i)}_{2}, D e_j \mathbf{t}^{(j)}_{2}   \rangle_{g_{\Upsilon_{i,r}}}
\end{bmatrix} \sqrt{\det({g_{\Upsilon_{i,r}}}}) du_1 du_2,\label{Siffnessmatrix}
\EEA
where the Riemannian inner product operations in each component will be defined according to the co-domain of $D$.
In the case of {\color{black}the} Bochner Laplacian, the Riemannian inner product is over {\color{black}a} $(2,0)$ tensor field, whereas for the Hodge Laplacian, it will be defined over $(2,0)$ tensor fields and functions for the first and second components of the co-domain, respectively.

We should again point out that the formulation above is only useful when the true local parameterization $\Upsilon_{i,r}$ is known. In our case, we consider using the local curved mesh parameterization $\tilde{\Phi}_{C_{i,r}}$ as defined in \eqref{PhitildeC}. We also note that we choose to employ such a parameterization over the flat local mesh parameterization $\tilde{\Phi}_{T_{i,r}}$ used in the previous section, since the latter introduces significant errors in the calculation of the stiffness matrix with zero Christoffel symbols.  

With this parameterization, we effectively solve \eqref{variationalBochnereigvalprob} in the following space
\BEA
\widehat{\mathcal{W}}_N = \left\{ W = \sum_{i=1}^N\sum_{\ell=1}^d W^{(i)}_\ell e_i \mathbf{t}^{(i)}_\ell,\, e_i: \tilde{\Phi}_{C_{i,r}}(T) \to \BR\right\},\label{WNhat}
\EEA
which differs from $\mathcal{W}_N$ on the parameterized domain of $e_i$. We will denote the resulting eigenvalue problem as
\BEA
\mathbf{\widehat{A}} \mathbf{W} = \lambda \mathbf{\widehat{B}} \mathbf{W},\label{gen-eigvalprob2}
\EEA
where $\mathbf{W}$ is defined as in \eqref{vecW} except that {\color{black}$W \in \widehat{\mathcal{W}}_N$.} Also, the symmetrized stiffness and mass matrices $\mathbf{\widehat{A}}$ and $\mathbf{\widehat{B}}$ are defined exactly as $\mathbf{{A}}$ and $\mathbf{{B}}$, respectively, where the local curved mesh parameterization $\tilde{\Phi}_{C_{i,r}}$ is used in place of $\Upsilon_{i,r}$ in \eqref{Massmatrix} and \eqref{Siffnessmatrix}.


Detailed computations for the mass matrix for vector fields, as well as for the stiffness matrices for the Bochner and Hodge Laplacians in the case of the local curved mesh method, are provided in the Appendices \ref{mass-matrix-computation-appendix}, \ref{bochner-laplacian-computation-appendix}, and \ref{hodge-laplacian-computation-appedix}, respectively. The remaining details for computing the mass matrix are summarized in the following remark.

\begin{remark}
\label{parallel_transport_remark}
A subtlety in the above block formula is in computing $\left\langle \mathbf{t}^{(i)}_{k_1}, \mathbf{t}^{(j)}_{k_2} \right\rangle_{g_{\Phi_{C_{i,r}}}}$ since the local parameterizations of these pair of tangent vectors are different. In particular, since $\mathbf{t}^{(i)}_{k_1}$ is constructed at the base point $x_i$ and $\mathbf{t}^{(j)}_{k_2}$ is constructed at the base point $x_j$, we need to represent these two vectors with the same coordinates to have a well defined inner product. To accomplish this, we first represent $\mathbf{t}^{(j)}_{k_2}$ in terms of $\{\mathbf{t}^{(i)}_{1}, \mathbf{t}^{(i)}_2, \mathbf{t}^{(i)}_3\}$: 
$$
\mathbf{t}^{(j)}_{k_2} = \sum_{k=1}^3 b_k \mathbf{t}^{(i)}_{k} 
$$
Let $\vec{b} = (b_1, b_2, b_3)^\top$. Similarly, let $\vec{a}$ denote a vector of all zeros with one in the $k_1$-th position, which is the analogous representation for $\mathbf{t}^{(i)}_{k_1}$. We then regress the above vectors onto the natural basis for the tangent space corresponding to the parameterization, given by $\{ \partial_{u_1} \Phi_{C_{i,r}}, \partial_{u_2}\Phi_{C_{i,r}} \}$ to obtain coefficient vector $\mathbf{a}^{ii}_{k_1}, \mathbf{a}^{ij}_{k_2}$ satisfying
\BEA
\begin{aligned}
\vec{a} &=& (\mathbf{a}^{ii}_{k_1})_1 \partial_{u_1} \Phi_{C_{i,r}} + (\mathbf{a}^{ii}_{k_1})_2 \partial_{u_2} \Phi_{C_{i,r}},  \\
\vec{b}  &=& (\mathbf{a}^{ij}_{k_2})_1 \partial_{u_1} \Phi_{C_{i,r}} + (\mathbf{a}^{ij}_{k_2})_2 \partial_{u_2} \Phi_{C_{i,r}}. \notag
\end{aligned}
\EEA
The two linear problems above can be written as follows
\BEA
\mathbf{T}^\top \mathbf{t}^{(i)}_{k_1} &=&  \mathbf{R}_{i,r}  \mathbf{a}_k^{ii}, \notag \\
\mathbf{T}^\top \mathbf{t}^{(j)}_{k_2} &=& \mathbf{R}_{i,r}  \mathbf{a}_k^{ij},  \notag
\EEA
where $\mathbf{T} = [\mathbf{t}^{(i)}_{1}, \mathbf{t}^{(i)}_{2}, \mathbf{t}^{(i)}_3 ] \in \BR^{3\times 3}$ and 
$\mathbf{R}_{i,r}=[\partial_{u_1}, \partial_{u_2}] \in \BR^{3\times 2}$ whose columns form a tangent space for the local curved mesh of triangle $T_{i,r}$, as defined in \eqref{cmm-basis-equation}. The regression solutions of these two problems are given by
\BEA  
\mathbf{a}_k^{i\ell} = ( \mathbf{R}_{i,r}^\top \mathbf{R}_{i,r})^{-1}\mathbf{R}_{i,r}^\top \mathbf{T}^\top \mathbf{t}^{(\ell)}_k, \notag
\EEA
for both $\ell = i$ and $\ell= j$.

Since $\mathbf{a}_k^{ii}$ and $\mathbf{a}_k^{ij}$ are the components of vectors $ \mathbf{t}^{(i)}_{k_1}$ and $ \mathbf{t}^{(i)}_{k_2}$, respectively, under the same coordinates $\{\partial_{u_1}, \partial_{u_2}\}$, then it is clear that
\BEA 
\label{cmm-vector-field-innerproduct}
\langle \mathbf{t}^{(i)}_{k_1} , \mathbf{t}^{(j)}_{k_2} \rangle_{g_{\Phi_{C_{i,r}}}}  = (\mathbf{a}^{ii}_{k_1})^\top g_{\Phi_{C_{i,r}}} \mathbf{a}^{ij}_{k_2}.
\EEA 

\end{remark}

\section{Convergence of Eigenvalues}
\label{theoretical-guarantees-section}

\comment{\color{red} 
Roadmap of the proof: With the notation above, I think part of 4.1 and 4.2 below will be needed to proof the following.

Suppose that $\lambda_\ell$ and  $\lambda_{\ell,N}$ denote spectra of \eqref{Bocheigprob} and \eqref{variationalBochnereigvalprob}, respectively. Then I would like to use standard technique in FEM to prove,
\[
\lambda_{\ell} \leq \lambda_{\ell,N} \leq \lambda_{\ell} + O(h_{X,M}).
\]

Finally, if we denote $\widehat{\lambda}_{\ell,N}$ as the corresponding spectrum of \eqref{gen-eigvalprob2}, then we can use perturbation theory of matrices to deduce something like, 
\[
|\widehat{\lambda}_{\ell,N} -\lambda_{\ell,N} | \leq C_1 \|\mathbf{A}- \widehat{\mathbf{A}}\| + C_2\|\mathbf{B}- \widehat{\mathbf{B}}\|
\] 

}

In this section, we study the convergence of the eigenvalues of a Bochner Laplacian discretization, obtained from the generalized eigenvalue problem \eqref{gen-eigvalprob2} using the local curved mesh method, to the true eigenvalues of the Bochner Laplacian \eqref{Bocheigprob}.

In Section \ref{metric-error-quantification-subsection}, we introduce some geometric preliminaries and define local charts {\color{black}that} serve as the analytic {\color{black}objects estimated} by the local curved mesh method. We also state convergence results that follow from standard GMLS results adapted to our setup. In Section~\ref{variational-subsection}, we {\color{black}discuss} the spectral error due to variational approximation on $\mathcal{W}_N$. The main result in this section is the spectral consistency of \eqref{gen-eigvalprob} {\color{black}with} \eqref{Bocheigprob}, stated in Lemma~\ref{spectralerrorvar}. In Section \ref{local-vector-field-subsection}, we quantify the spectral error between \eqref{gen-eigvalprob} and \eqref{gen-eigvalprob2} induced by the approximation of the local chart, {\color{black}as formalized} in Lemma~\ref{spectralerrorgeneigval}. In Section~\ref{main-subsection}, we state the main theorem based on Lemmas~\ref{spectralerrorvar} and \ref{spectralerrorgeneigval}, {\color{black}together with a discussion of} the spectral gap condition for convergence.

Throughout this section, we let $X=\{x_1, \dots, x_N \}$ denote a dataset of $N$ points sampled uniformly from a manifold $M$. We remark that, for simplicity, we assume that the estimated tangent space $\overline{T_{x_i}M}$ is exact {\color{black}at each point $x_i$}. In practice, high-order approximations of the tangent space can be obtained, and {\color{black}the} other error terms investigated in this analysis dominate. As in previous sections, we assume that the manifold {\color{black}has dimension $d=2$ and is} embedded in $\mathbb{R}^3$ to simplify {\color{black}the presentation.} However, generalizing the computations {\color{black}in} this section to higher dimensional manifolds of higher codimension is straightforward.

\subsection{Local metric approximation}\label{metric-error-quantification-subsection} 

Fix a point $x_i$ and consider \textcolor{black}{its} $k$-nearest neighbors
$
K(i)=\{x_{i_1},\dots,x_{i_k}\},
$
as well as $T(i)$, which is $K(i)$ projected to $T_{x_i}M$. Let
$
r_i=\max_{v_j\in T(i)}\|v_j\|.
$
In the previous equation and in the remainder of this paper, $\|\cdot\|$ denotes the Euclidean norm
\textcolor{black}{(for vectors in $\BR^2$, in the above setting)}.
We note that $k$, the number of nearest neighbors, \textcolor{black}{is allowed to depend on $N$,  $k=k(N)$}, and is chosen to scale sufficiently slow such that $r_i\to0$ as $N\to\infty$.
To this end, we assume that $N$ is sufficiently large so that
$
r_i\le r_M,
$
where $r_M$ denotes the injectivity radius of the manifold.

To facilitate the analysis below, we define the local fill and separation distances, respectively:
\[
h_i:=\sup_{v\in B_{0}(r_i)}\min_{v_j\in T(i)}\|v-v_j\|,
\qquad
q_i:=\frac12\min_{\substack{v_j,v_k\in T(i)\\ j\neq k}}\|v_j-v_k\|.
\]
Consider also the global fill and separation distances, respectively:
\[
h_{X,M}:=\sup_{x\in M}\min_{x_j\in X}\|x-x_j\|,
\qquad
q_X:=\frac12\min_{\substack{x_i,x_j\in X\\ i\neq j}}\|x_i-x_j\|.
\]

We assume that the dataset $X$ is quasi-uniform in the sense that
\BEA
q_X \le h_{X,M} \le c_q\,q_X,\notag
\EEA
for some constant $c_q>0$. It is easy to see that {\color{black}the projected data are also quasi-uniform in the same sense} ( i.e., $q_i \leq h_i \leq c_{q_i} q_i$, for some constant $c_{q_i}>0$). Since the points $\vec{v} \in B_{\vec{0}}(r_i)$ correspond to projected points in $M$, one can verify that
 \textcolor{black}{
\BEA
h_i \textcolor{black}{\le C_{\text{loc}}\,h_{X,M}} \quad \text{for all } i\in\{1,\dots,N\},
\label{local-global-fill-distance-bound}
\EEA
for some constant $C_{\text{loc}}\ge1$ (often $C_{\text{loc}}=1+O(r_i^2)$ in smooth small-radius regimes with the constant depends on the curvature of the manifold). One can verify this inequality by using the fact that Euclidean distance of any two nearby points is an accurate approximation to the geodesic distance between these points and the regularity of the projection operator $\exp_{x_i}^{-1}:M \to T_{x_i}M$.}

\comment{\color{red}
Assume \(M\subset \mathbb{R}^n\) is \(C^2\), and work in a sufficiently small normal neighborhood of \(x_i\).
Then geodesic and ambient Euclidean distances are locally equivalent: there exists \(c>0\), depending on curvature, such that
\[
(1-c r_i^2)\|x-y\|
\le
d_M(x,y)
\le
(1+c r_i^2)\|x-y\|,
\]
for $(x,y)$ in that neighborhood.

Let \(\pi_i\) be the local chart map (e.g. orthogonal/tangent projection or \(\exp_{x_i}^{-1}\), depending on setup), with local distortion
\[
\|\pi_i(x)-\pi_i(y)\|
\le (1+c r_i^2) d_M(x,y),
\]
where the upper bound $\sup_{z \in B(x_i,r_i)}\|D\pi_i(z)\| \leq 1+cr_i^2$ is used.
Fix \(v\in B_0(r_i)\), and set \(x=\pi_i^{-1}(v)\).
By definition of global fill distance, there exists \(x_j\in X\) such that
\[
\|x-x_j\|\le h_{X,M}.
\]
Hence
\[
\|v-\pi_i(x_j)\| =
\|\pi_i(x)-\pi_i(x_j)\|
\le
(1+c r_i^2)d_M(x,x_j)
\le
(1+c r_i^2)^2\|x-x_j\|
\le
(1+c r_i^2)^2 h_{X,M}.
\]
Taking first \(\min\) over projected samples and then \(\sup\) over \(v\in B_0(r_i)\), we obtain
\[
h_i \le C h_{X,M},
\qquad
C=(1+c r_i^2)^2=1+O(r_i^2).
\]

}

Let $\gamma_{x_i}:B_{0}(r_i)\to\BR^3$ denote a local parameterization of the form
\BEA
\gamma_{x_i}(v_1,v_2)=x_i+(v_1,v_2,z_{x_i}(v_1,v_2)),
\label{analytic-local-parameterization}
\EEA
for some smooth map $z_{x_i}$, which is obtained by projecting nearby points onto $T_{x_i}M$.

In this section, we quantify the error in estimating the coordinates induced by $\gamma_{x_i}$ using the coordinates corresponding to $\alpha_{x_i}$ in \eqref{mls-parameterization-equation}. This result follows from general approximation error estimates for GMLS, which were derived in \cite{mirzaei2012generalized}. We state their results in our setting below.

\begin{lem}
\label{mls-approximation-error}
Let $X \subseteq M$ be such that $h_{X, M} \leq h_0$ for some constant $h_0 > 0$. Define $\Omega_i^* = :=\bigcup_{v\in B_0(r_i) {\color{black}\subset \mathbb{R}^2}} B_v(C_2h_0)\subset {\color{black}\mathbb{R}^2}$ for some constant $C_2 > 0$. Then for any $f \in C^{m+1}(\Omega_i^*)$ and a multi index $\ell$ that satisfies $|\ell| \leq m$, we have  
    \BEA
    |D^\ell f(\vec{v}) - D^\ell \hat{f} (\vec{v}) | \leq C_i h_i^{m+1-|\ell|} |f|_{C^{m+1}(\Omega_i^*)} \leq C h_{X,M}^{m+1-|\ell|} |f|_{C^{m+1}(\Omega_i^*)}, \notag
    \EEA
    for any $\vec{v} \in B_{\vec{0}}(r_i)$, where $\hat{f}$ is a GMLS approximation to $f$ using polynomial of up to degree-$m$. Here, $D^\ell$ is a shorthand for multivariate derivative of order-$\ell$. {\color{black}The second inequality follows directly from \eqref{local-global-fill-distance-bound} with $C = C_{\text{loc}}^{m+1-|\ell|}\max_i C_i>0$ that is independent of $h_i$ and $h_{X,M}$.}
\end{lem}

In our case, we {\color{black}approximate} the $d=2$-dimensional smooth function $z_{x_i}$ in \eqref{analytic-local-parameterization} {\color{black}by} a polynomial $p_i$ as in \eqref{local-polynomial}, of degree $m=2$, such that
\BEA
\left|D^\ell z_{x_i}(\vec{v}) - D^\ell p_i(\vec{v}) \right|
= O\left(h_{X,M}^{3-|\ell|}\right),
\label{GMLSerrorbdd}
\EEA
for any multi-index $\ell$ with $|\ell|\leq 2$. We remark that, since $z_{x_i}:B_{\vec{0}}(r_i)\to\mathbb{R}$ is smooth, it can be extended to a smooth function on $\Omega_i^*$.

Using this result, we can obtain bounds on estimating the necessary local quantities. The proof is simple, and can be found in Appendix \ref{proofs-of-technical-lemmas}.
\begin{lem}
\label{local-metric-approximation-error} Let $M \subset \mathbb{R}^3$ be a two-dimensional $C^3$ manifold. Denote by $\left\{ \partial_{v_1} \left(  \gamma_{x_i} \right) , \partial_{v_2} \left(\gamma_{x_i} \right) \right\}$ and $\left\{ \partial_{v_1} \left(  \alpha_{x_i} \right) , \partial_{v_2} \left(\alpha_{x_i} \right) \right\}$ the {\color{black}two sets of basis vectors} induced from $\gamma_{x_i}$ defined in \eqref{analytic-local-parameterization} and $\alpha_{x_i}$ defined in \eqref{mls-parameterization-equation}, respectively. 
     Denote the Riemannian metric associated with each set of {\color{black}basis vectors} by $g_{\gamma_{x_i}}$ and $g_{\alpha_{x_i}}$, and the components of their inverse as $g_{\gamma_{x_i}}^{qs}$ and $g_{\alpha_{x_i}}^{qs}$.
    Let $\Gamma^q_{st}(\gamma_{x_i})$,  $\Gamma^q_{st}\left(\alpha_{x_i}\right)$ denote the Christoffel symbols associated with $g_{\gamma_{x_i}}$, $g_{\alpha_{x_i}}$, respectively. Fix any triangle $T_{i,r} \in R(i)$. Then for any $\vec{v} \in T_{i,r}$, 
    \BEA 
    \left\| \partial_{v_k} \left(  \gamma_{x_i} \right)(\vec{v}) - \partial_{v_k} \left(  \alpha_{x_i} \right)(\vec{v})  \right\| &=& O(h_{X, M}^2), \quad k=1,2, \notag \\
  \left| \sqrt{\det(g_{\gamma_{x_i}})(\vec{v})} -\sqrt{\det(g_{\alpha_{x_i}}(\vec{v}) )} \right| &=& O({\color{black}h_{X, M}^2}), \notag \\
    \left| \left( g^{qs}_{\gamma_{x_i}} \right)(\vec{v}) -  \left( g^{qs}_{\alpha_{x_i}} \right)(\vec{v}) \right| &=& O({\color{black}h_{X, M}^2}), \notag \\
    \left|\Gamma^q_{st}\left( \gamma_{x_i} \right)( \vec{v} ) -   \Gamma^q_{st}(\alpha_{x_i})(\vec{v} ) \right| &=& O(h_{X, M}) \notag,
    \EEA
    where the constant terms in the big-oh notations above depend on the $C^3$ norm of $\alpha_{x_i}$.
\end{lem}

{\color{black}We note that the approximation error for the Christoffel symbols is one order lower than that for the other terms because they involve second derivatives, corresponding to $|\ell|=2$ in \eqref{GMLSerrorbdd}.}

\subsection{Variational approximation}\label{variational-subsection}

The results in the previous subsection quantify the error between estimating the metric and its derivatives. Here, we quantify the error arising from the estimation of a vector-field $W$ with elements of the space $\mathcal{W}_N$ defined in \eqref{WN}. 

To this end, consider a vector field $W: M \to TM$. While the computation is performed using the representation of the vector field under the basis  $\{ \partial_{u_1} \Phi_{C_{i,r}}, \partial_{u_2} \Phi_{C_{i,r}} \}$ evaluated at $(\Phi_{i,r})^{-1}(\vec{v}_{i_{j_r}})$ as discussed in Remark~\ref{parallel_transport_remark}, it is theoretically more convenient to work in the tangent space $T_{x_i}M$, leveraging the error bounds induced by the estimated coordinate system $\left\{ \partial_{v_1} (\alpha_{x_i}), \partial_{v_2} (\alpha_{x_i}) \right\}$ as deduced in Lemma~\ref{local-metric-approximation-error}. Analytically, due to the diffeomorphism outlined in Section \ref{metric-error-quantification-subsection}, any vector field $W:M \to TM$ can be written in the local coordinates as 
\BEA 
W(v_1, v_2) = c^{(i,r)}_1(v_1,v_2) \partial_{v_1} \left(\gamma_{x_i}\right) + c^{(i,r)}_2(v_1,v_2) \partial_{v_2} \left(\gamma_{x_i}\right).  \notag 
\EEA

{
\begin{defi} 
Define the local interpolation operator $I_N: H^1(TM) \to \mathcal{W}_N$, where $\mathcal{W}_N$ is defined in \eqref{WN}, as follow: 
\BEA
I_N W = {\color{black}I_N \left(c^{(i,r)}_1\right)\partial_{v_1} \left(\gamma_{x_i}\right) + I_N \left(c^{(i,r)}_2\right) \partial_{v_2} \left(\gamma_{x_i}\right)}.\notag 
\EEA
Here, we abused the notation $I_N$ to also act on any function $f$,
\[
I_N f = \sum_{i} f(x_i)e_i,
\]
where $\{e_i\}$ is the nodal basis.
\end{defi}
With this definition, we can show {\color{black}the following:}
\begin{lem}{\label{interpolationerror}}
For any $W \in H^2(TM)$, we have,
\[
\left|I_NW - W\right|_{H^{s}(TM)} \leq c h_{X,M}^{2-s} |W|_{H^2(TM)},
\]
for {\color{black}$s=0,1$} and some constant $c>0$ that can depend on {\color{black}$q_{X}$} and the regularity of ${\color{black}g_{\gamma_{x_i}}}$. 
\end{lem}

\begin{proof} 
First, let us derive a bound for the norm of $D (W - I_NW)$ represented in ambient coordinates. For any {\color{black}$(v_1,v_2)\in T_{i,r}$},
\BEA 
    \| D (W{\color{black}(v_1,v_2)} - I_NW{\color{black}(v_1,v_2)})  \|^2 &=& \left\| \sum_{j=1}^n \left( D (W{\color{black}(v_1,v_2)} - I_NW{\color{black}(v_1,v_2)} ) \cdot \mathbf{e}_j \right) \mathbf{e}_j\right\|^2 \notag \\
    &=&  \sum_{j=1}^n \left( \sum_{k=1}^2 D(c^{(i,r)}_k -I_N c^{(i,r)}_k) \partial_{v_k}(\gamma_{x_i}) \cdot \mathbf{e}_j + (c^{(i,r)}_k - I_Nc^{(i,r)}_k) D \partial_{v_k}(\gamma_{x_i}) \cdot \mathbf{e}_j \right)^2 \notag\\
    &\leq& 4 \sum_{j=1}^n \sum_{k=1}^2  \left(D(c^{(i,r)}_k -I_N c^{(i,r)}_k) \partial_{v_k}(\gamma_{x_i}) \cdot \mathbf{e}_j\right)^2 + \left(c^{(i,r)}_k - I_Nc^{(i,r)}_k) D\partial_{v_k}(\gamma_{x_i}) \cdot \mathbf{e}_j \right)^2 \notag\\
    &\leq& 4nC \sum_{k=1}^2  \left(D(c^{(i,r)}_k -I_N c^{(i,r)}_k)\right)^2 + \left(c^{(i,r)}_k - I_Nc^{(i,r)}_k) \right)^2, \notag
\EEA 
where we used Young's inequality to bound the mixed terms and set $C>0$ to be an upper bound for the basis expansion. {\color{black}We should point out
in the right-hand terms above, the coefficients and basis vectors are understood to be evaluated at $(v_1,v_2)$. Particularly,
$c_k^{(i,r)} = c_k^{(i,r)}(v_1,v_2), I_N c_k^{(i,r)} = I_N c_k^{(i,r)}(v_1,v_2)$ and  $\partial_{v_k}(\gamma_{x_i}) = \partial_{v_k}(\gamma_{x_i})|_{(v_1,v_2)}$, which are suppressed to avoid stating lengthy formulas.}

Here, we deduce the case of {\color{black}$s=1$} since the other case follows by a similar argument.
\BEA
\left|I_NW - W\right|_{H^1(TM)}^2 &=& \int_{M} \langle D(I_NW - W),D(I_NW - W) \rangle_g d\text{Vol}  \notag \\
&=& \sum_{T_{i,r} \in R(i)} \int_{T_{i,r}} \langle D(I_NW - W),D(I_NW - W) \rangle_{{\color{black}g_{\gamma_{x_i}}}} \sqrt{\text{det}({\color{black}g_{\gamma_{x_i}}})}\, dv_1\,dv_2 \notag\\
&\leq& g_{max} \sum_{T_{i,r} \in R(i)} \int_{T_{i,r}} \langle D(I_NW - W) ,D(I_NW - W) \rangle_{{\color{black}g_{\gamma_{x_i}}}}\, dv_1\,dv_2 \notag \\
&\leq& 4nCg_{max} \sum_{T_{i,r} \in R(i)}  \sum_{k=1}^2 \int_{T_{i,r}}  \left(D(c^{(i,r)}_k -I_N c^{(i,r)}_k)\right)^2 + \left(c^{(i,r)}_k - I_Nc^{(i,r)}_k) \right)^2  \, dv_1\,dv_2, \label{eq1proof_lem4.3}
\EEA
where we have used the local parameterization defined in \eqref{analytic-local-parameterization} in the second equality, and the regularity of the metric tensor {\color{black}$g_{\gamma_{x_i}}$ in the third line above, that is, $g_{min} I \preceq {\color{black}g_{\gamma_{x_i}}} \preceq g_{max} I$, where $0<g_{min} \leq g_{max} < \infty$.} 

Consider the following standard finite element result. For any interpolant $I_Nf$ by polynomials of degree at most $m-1$ of a function $f \in W^{m,p}(T_{i,r})$, we have the following local interpolation error bound (see e.g. Theorem 4.4.4, in \cite{brenner2008mathematical}):
\[
\left|I_Nf - f \right|_{W^{j,p}(T_{i,r})} \leq c h_i^{m-j} q_i^{j} \left|f \right|_{W^{m,p}(T_{i,r})}
\]
$1\leq p\leq \infty$, $0\leq j \leq m$. It follows that, for $m = 2$, $p=2$, $j=1$, we have
\[
\left|I_Nf - f \right|_{H^{1}(T_{i,r})} \leq c h_i q_i |f|_{H^{2}(T_{i,r})},
\]
Inserting this bound to \eqref{eq1proof_lem4.3}, we obtain, 
\BEA
\left|I_NW - W\right|_{H^1(TM)} &\leq & C\sqrt{ng_{max}} h_{X,M}  \left(\sum_{T_{i,r} \in R(i)}\sum_{k=1}^2 \int_{T_{i,r}} (D^2 {\color{black}c_k^{(i,r)}})^2\,dv_1\,dv_2 \right)^{1/2} \notag \\ 
&\leq & C\sqrt{ng_{max}} h_{X,M}  \left(\sum_{T_{i,r} \in R(i)}\sum_{j=1}^n \sum_{k=1}^2 \int_{T_{i,r}} (D^2 {\color{black}c_k^{(i,r)}} \partial_{v_k}(\gamma_{x_i}) \cdot \mathbf{e}_j )^2\,dv_1\,dv_2 \right)^{1/2} \notag \\
 &\leq& C \frac{\sqrt{g_{max}}}{\sqrt{g_{min}}} h_{X,M} \left( \sum_{T_{i,r} \in R(i)} \int_{T_{i,r}} \langle D^2W, D^2W \rangle_{{\color{black}g_{\gamma_{x_i}}}} \sqrt{\text{det}({\color{black}g_{\gamma_{x_i}}})}\, dv_1\,dv_2\right)^{1/2} \notag \\
&\leq & C h_{X,M}  |W|_{H^2(TM)},\notag
\EEA
where we used the regularity of metric tensor ${\color{black}g_{\gamma_{x_i}}}$ and abuse the constant notation $C>0$ that are different after each inequality.
\end{proof}

For the technical analysis below, we define the Ritz projection.

\begin{defi}[Ritz projection] 
Let $R_N: H^1(TM) \to \mathcal{W}_N$ be the Ritz projection, which  is the orthogonal projection in the following sense:
\BEA
a(W - R_NW, V_N) = 0, \quad \forall V_N \in \mathcal{W}_N. \label{ortho}
\EEA
\end{defi}
Following the standard argument in a priori analysis for FEM solution to finding $W\in H^1(TM)$ such that $a(W,V) = b({\color{black}U},V)$ for {\color{black}some $U\in L^2(TM)$ and} all $V \in H^1(TM)$, with $R_NW$ be the solution in $\mathcal{W}_N$, one can deduce the well-known a priori error bound
\BEA
\|W - R_NW\|_{H^1(TM)} \leq C h_{X,M} |W|_{H^2(TM)}\label{aprioriH1}.
\EEA

For the $L^2$-norm error bound, consider the following auxiliary problem {\color{black}of finding $\Phi \in H^1(TM)$, such that
$
 a(\Phi,V) = b(E,V)
$ for all $V \in H^1(TM)$, where $E = W - R_NW$. Using the symmetry of $a$ and choosing $V=E$, we arrive at}
\[
a(E,\Phi) = b(E,E) ={\color{black} \|E\|^2_{L^2(TM)}}. 
\]
Note that, 
\[
\|E\|_{L^2(TM)}^2  = a(E, \Phi - I_N\Phi) \leq M \|E\|_{H^1(TM)} \|\Phi - I_N\Phi \|_{H^1(TM)} \leq C h_{X,M}^2 |W|_{H^2(TM)} \|\phi\|_{H^2(TM)} ,
\]
where we have used the orthogonality condition in \eqref{ortho} in the first equality, continuity of $a$ in the {\color{black}following} equality, and finally, the interpolation error in Lemma~\ref{interpolationerror} and the a priori error bound \eqref{aprioriH1}. By the standard regularity result for the auxiliary problem above, $\|\phi\|_{H^2(TM)} \leq C \|E\|_{L^2(TM)}$,
we conclude the $L^2$-error bound for the Ritz projection
\BEA
\|W - R_NW\|_{L^2(TM)} \leq C h_{X,M}^2 |W|_{H^2(TM)}. \label{Ritz_L2error}
\EEA 
With this preliminary result, we can deduce the following bound for the eigenvalues.

\begin{lem} \label{spectralerrorvar}
Let $\lambda_\ell$ and $\lambda_{\ell,N}$ be the $\ell$th eigenvalues of \eqref{Bocheigprob}
and \eqref{variationalBochnereigvalprob}, respectively. Then, there exists {\color{black}a constant $C_\ell = C_0 \lambda_\ell^2>0$>0, where $C_0$ is a constant that is independent of $h_{X,M}$ but can depend on the curvature of the manifold,} such that
\[
\lambda_\ell \leq \lambda_{\ell,N} \leq \lambda_{\ell} + C_\ell h_{X,M}^2. 
\]
\end{lem}

\begin{proof}
The first inequality above is clear since
\[
\lambda_\ell  = \min_{S_\ell \subset H^1(TM)}\max_{W\in S_\ell} \frac{a(W,W)}{b(W,W)}, \quad\quad \lambda_{\ell,N}  = \min_{S_\ell \subset \mathcal{W}_N}\max_{W\in S_\ell} \frac{a(W,W)}{b(W,W)},
\]
where $\text{dim}(S_\ell) = \ell$, and $\mathcal{W}_N \subset H^1(TM)$.

To prove the second inequality, let $E_\ell$ be the space spanned by the leading eigenvector fields, $\Phi_1, \ldots, \Phi_\ell$. Denote $E_{\ell,N} = R_N E_\ell$, where $R_N$ denotes the Ritz projection as defined above. Then,
\BEA
\lambda_{\ell, N} \leq \max_{W\in E_{\ell,N}} \frac{a(W,W)}{b(W,W)}  = \max_{W\in E_{\ell}} \frac{a(R_NW,R_NW)}{b(R_NW,R_NW)} \leq \max_{W\in E_{\ell}} \frac{ a(W,W)}{\|R_NW\|_{L^2(TM)}^2},\label{ineq_lemma44}
\EEA
where in the last inequality, we use the fact that $a(R_NW,R_NW) + a(R_NW-W,R_NW-W) =  a(W,W)$.  For $W \in E_\ell$,  \eqref{Ritz_L2error} becomes 
\[
\|W - R_NW\|_{L^2(TM)} \leq C h_{X,M}^2 |W|_{H^2(TM)} \leq C h_{X,M}^2 |D^2W|_{L^2(TM)} \leq C {\color{black}h_{X,M}^2}\lambda_\ell \|W\|_{L^2(TM)},
\]
and it is clear that
\[
\|R_NW\|_{L^2(TM)} \geq \| W\|_{L^2(TM)} - \| W - R_NW\|_{L^2(TM)} \geq \| W\|_{L^2(TM)} (1 - C h^2_{X,M}{\color{black}\lambda_\ell}).
\]
Inserting this to \eqref{ineq_lemma44}, we obtain
\[
\lambda_{\ell, N} \leq  \max_{W\in E_{\ell}} \frac{ a(W,W)}{\|W\|_{L^2(TM)}^2} {\color{black}\left(1+ 2C\lambda_\ell h^2_{X,M} + O(h^4_{X,M})\right) = \lambda_\ell  + O( h_{X,M}^2),}
\]
{\color{black}as $h_{X,M} \to 0$} and the proof is completed. {\color{black}In the statement of the theorem, we note that the constant $C_\ell>0$ in the big-oh depends on $\lambda_\ell$ quadratically.}
\end{proof}

\subsection{Local vector field approximation}\label{local-vector-field-subsection}

In this section, we derive a bound between the eigenvalues of the generalized eigenvalue problems in \eqref{gen-eigvalprob} and \eqref{gen-eigvalprob2}. To achieve this, we first derive absolute error between the entries of $\mathbf{A}$ and $\mathbf{\widehat{A}}$ and between the entries of of $\mathbf{B}$ and $\mathbf{\widehat{B}}$, 
which amount to bounding the errors between the entries of $\mathbf{S}$ and $\mathbf{\widehat{S}}$ and between the entries of $\mathbf{M}$ and $\mathbf{\widehat{M}}$
. 

Technically, we need to bound the error between the estimators $W_N \in \mathcal{W}_N$ and $\widehat{W}_N \in \mathcal{\widehat{W}}N$. For convenience
{\color{black} in the following discussion, let us restate the local representation of the underlying vector field $W \to TM$ at any $(v_1,v_2)\in T_{i,r}$:
\BEA 
W(v_1, v_2) = c^{(i,r)}_1(v_1,v_2) \partial_{v_1} \left(\gamma_{x_i}\right) + c^{(i,r)}_2(v_1,v_2) \partial_{v_2} \left(\gamma_{x_i}\right),  \notag 
\EEA
}
and represent $W_N$ and  $\widehat{W}_N$ on $T_{i,r}$ as follows,
\BEA 
 W_N(v_1,v_2) &=& b^{(i,r)}_1(v_1,v_2) \partial_{v_1} (\gamma_{x_i})  + b^{(i,r)}_2(v_1,v_2) \partial_{v_2} (\gamma_{x_i}), \label{WNloc}\\
 \widehat{W}_N(v_1,v_2) &=& \widehat{b}^{(i,r)}_1(v_1,v_2) \partial_{v_1} (\alpha_{x_i})  + \widehat{b}^{(i,r)}_2(v_1,v_2) \partial_{v_2} (\alpha_{x_i}), \label{approx_vf}  
\EEA
where the estimated coefficients {\color{black}$b^{(i,r)}_k$ and $\widehat{b}^{(i,r)}_k$} are functions of $T_{i,r}$.
{\color{black}We note that $b^{(i,r)}_k$ are the coefficients corresponding to the nodal basis interpolation of $W$ with respect to the representation using $\{ \partial{v_1} (\gamma_{x_i}), \partial_{v_2} (\gamma_{x_i}) \}$, as defined in \eqref{WNloc}.
On the other hand, the coefficients $\widehat{b}^{(i,r)}_k$ correspond to the nodal basis interpolation of $W$ with respect to the representation using the approximate local basis $\{ \partial{v_1} (\alpha_{x_i}), \partial_{v_2} (\alpha_{x_i}) \}$, as in \eqref{approx_vf}.}

Here, we recall that $\gamma_{x_i}$ and $\alpha_{x_i}$ denote the true local parameterization in \eqref{analytic-local-parameterization} and the curved mesh parameterization in \eqref{mls-parameterization-equation}, respectively.

\begin{lem}
\label{local-vector-field-approximation-error}
Let $M$ be a two-dimensional $C^3$-manifold. For a fixed $x_i \in X$, consider any triangle $T_{i,r} \in R(i)$. Denote by $W_N \in \mathcal{W}_N$ in \eqref{WNloc} as an approximation to the solution $W$ of \eqref{variationalBochnereigvalprob} in the analytic charts induced by $\gamma_{x_i}$. Let $\widehat{W}_N$ denote the estimate as defined in \eqref{approx_vf}, where the parameterizations $\{ \alpha_{x_i} \}_{i=1}^N$ are estimated using polynomials of degree$-2$ as in \eqref{local-polynomial}. Then
   \BEA 
\left\| D^\ell W_N -  D^\ell \widehat{W}_N \right\|_{L^\infty(T_{i,r})} = O({\color{black}h_{X, M}^{2-|\ell|}}), \notag
   \EEA 
   where $h_{X,M}$ is the global fill distance defined as before and $\ell$ is an index of order $|\ell| \leq 1$, and the constant depends on the local separation distance $q_i$.  
\end{lem}

\begin{proof}
For any $(v_1,v_2) \in T_{i,r}$, 
we introduce an intermediate estimator:
\BEA 
 \widetilde{W}_N(v_1, v_2) &=& \widehat{b}^{(i,r)}_1(v_1,v_2) \partial_{v_1} \left(\gamma_{x_i}\right) + \widehat{b}^{(i,r)}_2(v_1,v_2) \partial_{v_2} \left(\gamma_{x_i}\right).  \notag 
\EEA 
We split the error as follows:
\BEA 
\left\| D^\ell W_N(v_1,v_2) -  D^\ell \widehat{W}_N(v_1,v_2) \right\| &\leq& \left\| D^\ell W_N(v_1,v_2) - D^\ell \widetilde{W}_N(v_1,v_2) \right\| + \left\| D^\ell \widetilde{W}_N(v_1,v_2) - D^\ell \widehat{W}_N(v_1,v_2) \right\|. \notag
\EEA 
The error of the first term can be analyzed in the basis of the ambient space: 
\BEA 
    \left\| D^\ell W_N(v_1,v_2) - D^\ell \widetilde{W}_N(v_1,v_2) \right\| &=& \left\| \sum_{j=1}^n \left( \left(D^\ell W_N(v_1,v_2) - D^\ell \widetilde{W}_N(v_1,v_2) \right) \cdot \mathbf{e}_j \right) \mathbf{e}_j\right\| \notag \\
    &=& \sqrt{ \sum_{j=1}^n \left( \sum_{k=1}^d D^\ell (b^{(i,r)}_k - \widehat{b}^{(i,r)}_k) \partial_{v_k}(\gamma_{x_i}) \cdot \mathbf{e}_j + (b^{(i,r)}_k - \widehat{b}^{(i,r)}_k) D^\ell \partial_{v_k}(\gamma_{x_i}) \cdot \mathbf{e}_j \right)^2. } \notag 
\EEA 
For simplicity of notation, we have suppressed the evaluation at $(v_1,v_2)$ above, as we do throughout what follows. We note that, for quasi-uniform data, we have $h_i/q_i \leq c_{q_i}$.

Clearly, it suffices to bound $| D^\ell ( b^{(i,r)}_k - \widehat{b}^{(i,r)}_k )|$. 
Let $\vec{v}_{i_{j_r}}$ denote an arbitrary vertex in $T_{i,r}$. Define the following:
\BEA 
\mathbf{D} = \Bigg[\partial_{v_1}(\gamma_{x_i}) (\vec{v}_{i_{j_r}}),  \partial_{v_2}(\gamma_{x_i}) (\vec{v}_{i_{j_r}})\Bigg], & \mathbf{\widehat{D}} = \Bigg[\partial_{v_1}(\alpha_{x_i}) (\vec{v}_{i_{j_r}}), \partial_{v_2}(\alpha_{x_i}) (\vec{v}_{i_{j_r}})\Bigg], \notag \\ 
\mathbf{b} = \Bigg(b^{(i,r)}_1 (\vec{v}_{i_{j_r}}), b^{(i,r)}_2 (\vec{v}_{i_{j_r}})\Bigg)^\top, & \mathbf{\widehat{b}} = \Bigg(\widehat{b}^{(i,r)}_1 (\vec{v}_{i_{j_r}}), \widehat{b}^{(i,r)}_2 (\vec{v}_{i_{j_r}})\Bigg)^\top. \notag 
\EEA 
We remark that $W(\vec{v}_{i_{j_r}}) = \mathbf{D} \mathbf{b}$ whereas $W(\vec{v}_{i_{j_r}}) = \mathbf{\widehat{D}} \mathbf{\hat{b}} + \hat{\mathbf{r}}$, where the first regression problem has no residual as ${\color{black}\mathbf{b}} = \mathbf{c} = \left(c^{(i,r)}_1 (\vec{v}_{i_{j_r}}), c^{(i,r)}_2 (\vec{v}_{i_{j_r}})\right)^\top$ on the vertex of the $T_{i,r}$, whereas the second regression problem may have a residual due to the approximation of the basis representation. Since $ \| \partial_{v_k}(\gamma_{x_i}) - \partial_{v_k} (\alpha_{x_i}) \| = O(h^2_i)$ (by Lemma~\ref{local-metric-approximation-error}) , it follows that $\| \mathbf{D} - {\color{black}\widehat{\mathbf{D}}} \|_2 \leq \| \mathbf{D} - \mathbf{\widehat{D}} \|_F = O(h^2_i)$, where $ \| \cdot \|_2 $ and $ \| \cdot \|_F $ denote the spectral and Frobenius matrix norms, respectively. Since the condition number of $\kappa_2({\color{black}\mathbf{D}})= O(1)$, for any small perturbation  $h^2_i  \ll 1$, it follows from the standard least-squares perturbation results (e.g., see Theorem~3.4 of \cite{demmel1997applied}) 
that $\|\mathbf{b} - \mathbf{\widehat{b}}\| \leq h^2_i 2\kappa_{2}(\mathbf{D})\|\mathbf{b}\|$. Hence, 
\BEA 
\left|b^{(i,r)}_k ( \vec{v}_{i_{j_r}}) - \widehat{b}^{(i,r)}_k ( \vec{v}_{i_{j_r}} )\right| = O(h^2_i). \notag  
\EEA 
Since $\vec{v}_{i_{j_r}}$ is an arbitrary vertex of $T_{i,r}$, it follows immediately from the definition of the hat functions that 
\BEA 
\left|D\left(b^{(i,r)}_k (v_1,v_2) - \widehat{b}^{(i,r)}_k (v_1,v_2)\right)\right| = O(h_i), \quad \forall (v_1,v_2)\in T_{i,r}, \notag
\EEA 
where the above constant in the big-oh notation depends on the condition number of $\mathbf{D}$. Hence, 
\BEA 
\left\| D^\ell W_N(v_1,v_2) - D^\ell \widetilde{W}_N(v_1,v_2) \right\| = O({\color{black}h_i^{2-|\ell|}}).  \notag 
\EEA 
For the second term $ \left\|D^\ell \widetilde{W}_N(v_1,v_2) - D^\ell \widehat{W}_N(v_1,v_2) \right\|$, we again analyze the error on the basis of the ambient space:
\BEA 
\left\| D^\ell \widetilde{W}_N(v_1,v_2) - D^\ell \widehat{W}_N(v_1,v_2) \right\| &=& \left\| \sum_{j=1}^n \left( (D^\ell \widetilde{W}_N(v_1,v_2) - D^\ell \widehat{W}_N(v_1,v_2) ) \cdot \mathbf{e}_j \right) \mathbf{e}_j\right\| \notag \\
    &=& \sqrt{ \sum_{j=1}^n \left( \sum_{k=1}^d (D^\ell \widehat{b}^{(i,r)}_k) ( \partial_{v_k}(\gamma_{x_i}) - \partial_{v_k}(\alpha_{x_i}) ) \cdot \mathbf{e}_j + \widehat{b}^{(i,r)}_k D^\ell ( \partial_{v_k}(\gamma_{x_i}) - \partial_{v_k}(\alpha_{x_i}) ) \cdot \mathbf{e}_j \right)^2. } \notag 
\EEA 
We proved in Lemma \ref{local-metric-approximation-error} that $\left\| D^\ell ( \partial_{v_k}(\gamma_{x_i}) - \partial_{v_k}(\alpha_{x_i}) )\right\| = O(h_i^{2-|\ell|})$ for $|\ell|=0,1$. 

Hence, 
\BEA 
\left\| D^\ell \widetilde{W}_N(v_1,v_2) - D^\ell \widehat{W}_N(v_1,v_2) \right\| = O({\color{black}h_i^{2-|\ell|}}), \notag 
\EEA 
where the constant in the big-oh notation depends on the $C^3$-norm of $\alpha_{x_i}$.
 Finally, since $h_i \leq C_{\text{loc}} h_{X,M}$, the proof is complete. 
\end{proof}

The following local integral errors can be deduced immediately from Lemmas~\ref{local-metric-approximation-error} and \ref{local-vector-field-approximation-error}. The proof can be found in the Appendix \ref{proofs-of-technical-lemmas}
\begin{lem}
\label{local-integral-convergence-rate}
    Denote by $V_N, W_N \in \mathcal{W}_N$ and ${\color{black}\widehat{V}_N, \widehat{W}_N} \in \mathcal{\widehat{W}}_N$, as defined in Lemma~\ref{local-vector-field-approximation-error}. Let $DW_N = (\nabla W_N)_{\gamma_{x_i}}$ and $D\widehat{W}_N = (\nabla \widehat{W}_N)_{\alpha_{x_i}}$ be represented in the coordinates $\gamma_{x_i}$ and $\alpha_{x_i}$, respectively.
    \comment{
    $$
    (\nabla W_N)_{\gamma_{x_i}} := g_{\gamma_{x_i}}^{sj} \Bigg(  \frac{\partial W_N^t}{\partial v_s} + W_N^p \Gamma^t_{ps}(\gamma_{x_i}) \Bigg) \frac{\partial}{ \partial v_t} \otimes \frac{\partial} {\partial v_j}.
    $$ 
    Similarly, 
    $$
    (\nabla \widehat{W}_N)_{\alpha_{x_i}} := g_{\alpha_{x_i}}^{sj} \Bigg(  \frac{\partial \widehat{W}_N^t}{\partial v_s} + \widehat{W}_N^p \Gamma^t_{ps}(\alpha_{x_i}) \Bigg) \frac{\partial}{ \partial v_t} \otimes \frac{\partial}{\partial v_j}.
    $$ }
    Then 
    \BEA 
 \left | \int_{T_{i,r}}  \left( \Bigg\langle (\nabla U_N)_{\gamma_{x_i}} , (\nabla W_N)_{\gamma_{x_i}}
 \Bigg\rangle_{g_{\gamma_{x_i}}} \sqrt{ \textup{det}\left(g_{\gamma_{x_i}} \right)} 
 -  \Bigg\langle  (\nabla \widehat{U}_N)_{\alpha_{x_i}}, (\nabla \widehat{W}_N)_{\alpha_{x_i}}
\Bigg\rangle_{g_{\alpha_{x_i}}} \sqrt{ \textup{det}\left(g_{\alpha_{x_i}} \right)} \right) dv_1dv_2 \right| = O({\color{black}h_{X,M}^3}). \label{lem4.6:eq1}
\EEA 
and
    \BEA 
 \left | \int_{T_{i,r}}  \left( \Bigg\langle (U_N)_{\gamma_{x_i}} , (W_N)_{\gamma_{x_i}}
 \Bigg\rangle_{g_{\gamma_{x_i}}} \sqrt{ \textup{det}\left(g_{\gamma_{x_i}} \right)} 
 -  \Bigg\langle  (\widehat{U}_N)_{\alpha_{x_i}}, (\widehat{W}_N)_{\alpha_{x_i}}
\Bigg\rangle_{g_{\alpha_{x_i}}} \sqrt{ \textup{det}\left(g_{\alpha_{x_i}} \right)} \right) dv_1dv_2 \right| = O({\color{black}h_{X,M}^4}). \label{lem4.6:eq2}
\EEA 
\end{lem}
{\color{black} 
Notice that the error rate in \eqref{lem4.6:eq2} is one order higher than that in \eqref{lem4.6:eq1} since it does not depend on the first derivatives or Christoffel symbols.}
To finalize the result, we review the following perturbation theory for the generalized eigenvalue problem {\color{black}(Theorem 3.2 in \cite{stewart1979pertubation}).} 

\begin{prop}\label{perturbtheory}Suppose $\mathbf{A}-\lambda \mathbf{B} \in \mathbb{R}^{2N\times 2N}$ is symmetric definite pencil with eigenvalues $\lambda_{1,N}\geq \lambda_{2,N} \geq \ldots \geq \lambda_{2N,N}$. Let $\mathbf{\widehat{A}}-\lambda \mathbf{\widehat{B}}$ be symmetric definite pencil with eigenvalues $\widehat{\lambda}_{1,N}\geq \widehat{\lambda}_{2,N} \geq \ldots \geq \widehat{\lambda}_{2N,N}$. Suppose that $\mathbf{\widehat{A}} = \mathbf{A} + \mathbf{E}_A$ and 
$\mathbf{\widehat{B}} = \mathbf{B} + \mathbf{E}_B$ satisfy
\[
\epsilon^2 := \|\mathbf{E}_A\|_2^2 + \|\mathbf{E}_B\|_2^2,\qquad {\color{black}\epsilon < c(\mathbf{A},\mathbf{B})}, 
\]
where 
\[
c(\mathbf{A},\mathbf{B}) := {\color{black}\sqrt{\min_{\|\mathbf{x}\|_2=1} (\mathbf{x}^\top \mathbf{A} \mathbf{x})^2 +(\mathbf{x}^\top \mathbf{B} \mathbf{x})^2}} > 0.
\]
is known as the Crawford number of the pencil $\mathbf{A}-\lambda \mathbf{B} $. Then,
\[
\frac{\left|\lambda_{\ell,N} -\widehat{\lambda}_{\ell,N} \right|}{\color{black}\sqrt{1 +\lambda_{\ell,N}^2}\sqrt{1+\widehat{\lambda}_{\ell,N}^2} } \leq \frac{\epsilon}{c(\mathbf{A},\mathbf{B})}.
\]
\end{prop}

The condition ${\color{black}\epsilon < c(\mathbf{A},\mathbf{B})}$ ensures that the perturbation size does not exceed the gap beyond which the perturbed matrix becomes indefinite.
\comment{Since $\mathbf{A}$ and $\mathbf{B}$ are symmetric positive definite, they are simultaneously diagonalized (see Corollary 8.7.2 in \cite{golub2013matrix}), and one can write $\lambda_{\ell,N} = \frac{a_\ell}{b_\ell}$, where $a_\ell$ and $b_\ell$ are eigenvalues of $\mathbf{A}$ and $\mathbf{B}$, respectively. Hence,
the spectral gap condition simply means,
\[
\epsilon^2 < \min_{\ell=1,\ldots, 2N} a_\ell^2+ b_\ell^2.
\]}
We can now finalize this section with the following result:

\begin{lem}\label{spectralerrorgeneigval}
Suppose {\color{black}$c(\mathbf{A},\mathbf{B}) > C N h_{X,M}^3$}, where the constant $C>0$ corresponds to the constant in the big-oh notation in Lemma~\ref{local-integral-convergence-rate}. Then
\[
\left|\lambda_{\ell,N} -\widehat{\lambda}_{\ell,N} \right| \leq \frac{\color{black}CN h_{X,M}^3}{c(\mathbf{A},\mathbf{B})} {\color{black}\sqrt{1 +\lambda_{\ell,N}^2}\sqrt{1+\widehat{\lambda}_{\ell,N}^2}}.
\]
\end{lem}

\begin{proof}
{\color{black}Let $U_N = e_i\mathbf{t}_k^{(i)}$ and $W_N = e_j\mathbf{t}_\ell^{(j)}$ for any $i,j\in\{1,\ldots, N\}$ and $k,\ell \in \{1,2\}$, where the hat functions $e_i$ and $e_j$ are defined in \eqref{WN}. Similarly, let $\widehat{U}_N = e_i\mathbf{t}_k^{(i)}$ and $\widehat{W}_N = e_j\mathbf{t}_\ell^{(j)}$ for any $i,j\in\{1,\ldots, N\}$ and $k,\ell \in \{1,2\}$, where the hat functions $e_i$ and $e_j$ are defined in \eqref{WNhat}. Recall that we can reparameterize the integrands in \eqref{lem4.6:eq1} in Lemma~\ref{local-integral-convergence-rate} using $\Phi_{i,r}: T \to T_{i,r}$ such that
\BEA 
G_{i,j, k,\ell}(T_{i,r}) &:=& \int_{T_{i,r}}  \Bigg\langle  (\nabla e_i\mathbf{t}_k^{(i)} )_{\gamma_{x_i}}, (\nabla e_j\mathbf{t}_\ell^{(j)})_{\gamma_{x_i}}
\Bigg\rangle_{g_{\gamma_{x_i}}} \sqrt{ \textup{det}\left(g_{\gamma_{x_i}} \right)} dv_1dv_2\notag  \\ &=&
\int_{T}  \Bigg\langle  (\nabla e_i\mathbf{t}_k^{(i)})_{\Upsilon_{i,r}}, (\nabla e_j\mathbf{t}_\ell^{(j)})_{\Upsilon_{i,r}}
\Bigg\rangle_{g_{{\Upsilon_{i,r}}}} \sqrt{ \textup{det}\left(g_{\Upsilon_{i,r}} \right)} du_1du_2,\notag
\\
\widehat{G}_{i,j,k,\ell}(T_{i,r}) &:=& \int_{T_{i,r}}  \Bigg\langle  (\nabla e_i\mathbf{t}_k^{(i)})_{\alpha_{x_i}}, (\nabla e_j\mathbf{t}_\ell^{(j)})_{\alpha_{x_i}}
\Bigg\rangle_{g_{\alpha_{x_i}}} \sqrt{ \textup{det}\left(g_{\alpha_{x_i}} \right)} dv_1dv_2\notag  \\ &=&
\int_{T}  \Bigg\langle  (\nabla  e_i\mathbf{t}_k^{(i)})_{\Phi_{C_{i,r}}}, (\nabla e_j\mathbf{t}_\ell^{(j)})_{\Phi_{C_{i,r}}}
\Bigg\rangle_{g_{{\Phi_{C_{i,r}}}}} \sqrt{ \textup{det}\left(g_{\Phi_{C_{i,r}}} \right)} du_1du_2,\notag
\EEA 
where $\Upsilon_{i,r}$ is 
defined in the discussion prior to Eq.~\eqref{function-stiffness-matrix-equation} and $\Phi_{C_{i,r}}$ is defined in \eqref{PhitildeC}. Then Eq.~\eqref{lem4.6:eq1} applied to these vector fields can be written as
\[
\left|G_{i,j,k,
\ell}(T_{i,r}) - \widehat{G}_{i,j,k,\ell}(T_{i,r})\right| = O(h^3_{X,M}),
\]
for all $i,j\in\{1,\ldots, N\}$ and $k,\ell \in \{1,2\}$. From the definition of the stiffness matrix in \eqref{Siffnessmatrix}, we have 
\[\left|\mathbf{S}_{2i-(2-k),2j-(2-\ell)} - \mathbf{S}_{2i-(2-k),2j-(2-\ell)}\right| \leq \sum_{T_{i,r}\in R(i)} \left|G_{i,j,k,
\ell}(T_{i,r}) - \widehat{G}_{i,j,k,\ell}(T_{i,r})\right| =O(h^3_{X,M}),\]
and} it is clear that  $\| \mathbf{S} - \widehat{\mathbf{S}}\|_2\leq \| \mathbf{S} - \widehat{\mathbf{S}}\|_F \leq CN{\color{black}h_{X,M}^3}$, {\color{black}where an additional factor $N$ arises from taking the Frobenius norm over a matrix of size $2N\times 2N$.} This implies that $\| \mathbf{E}_A\|_2 = \| \widehat{\mathbf{A}} - \mathbf{A}\|_2 = O({\color{black}N h_{X,M}^3})$. Using a similar argument, we also have, $\| \mathbf{E}_B\|_2 = \| \widehat{\mathbf{B}} - \mathbf{B}\|_2 = O({\color{black}Nh_{X,M}^4})$. By the assumption and Proposition~\ref{perturbtheory}, the proof is complete.
\end{proof}

\subsection{Main result} \label{main-subsection}

Collecting the assumptions and results from Lemma~\ref{spectralerrorvar} and \eqref{spectralerrorgeneigval}, we obtain.

{\color{black}
\begin{theo}\label{maintheorem} Let $M \subset \BR^3$ be a two-dimensional $C^3-$manifold. Let $\lambda_\ell$ and $\widehat{\lambda}_{\ell,N}$ be the $\ell$th eigenvalues of \eqref{Bocheigprob} and \eqref{gen-eigvalprob2}, respectively. \comment{Denote eigenvalues of $\mathbf{A}$ as $a_1\geq a_2 \geq \ldots \geq a_{2N}>0$ and eigenvalues of $\mathbf{B}$ as $b_1\geq b_2 \geq \ldots \geq b_{2N}>0$, where these matrices are defined in \eqref{gen-eigvalprob}.} Suppose that the following spectral gap condition holds, $c(\mathbf{A},\mathbf{B}) > CNh^3_{X,M}$,
where $h_{X,M}$ is the fill distance of a quasi-uniform point cloud data, $X= \{x_1,\ldots,x_N\} \subset M$ as defined in Section~\ref{metric-error-quantification-subsection}. Then
\[
\left|\lambda_\ell - \widehat{\lambda}_{\ell,N}\right| \leq \left|\lambda_{\ell} - \lambda_{\ell,N}\right| + \left|\lambda_{\ell,N} - \widehat{\lambda}_{\ell,N}\right| \leq C_{\text{bias}} h_{X,M}^2 + C_{\text{gap}}Nh_{X,M}^3,
\]
where the constant $C_{\text{bias}} = C_0\lambda_\ell^2>0$ is as shown in Lemma~\ref{spectralerrorvar}
and $C_{\text{gap}} = \frac{C_1}{c(\mathbf{A},\mathbf{B})}\sqrt{1 +\lambda_{\ell,N}^2}\sqrt{1+\widehat{\lambda}_{\ell,N}^2}>0$,
for some constants $C_0,C_1>0$ that are independent of $N$ and $h_{X,M}$, but depend on the curvature of the manifold.
\end{theo}
}

In the result above, one can interpret the result in Lemma~\ref{spectralerrorvar} as the bias error due to restricting the solution to the hypothesis space $\mathcal{W}_N$. On the other hand, the result in Lemma~\ref{spectralerrorgeneigval} corresponds to the manifold learning error.
 \comment{The spectral gap condition is described in terms of eigenvalues of $\mathbf{A}$ and $\mathbf{B}$ corresponds to the generalized eigenvalue problem,
\BEA
\mathbf{A} \mathbf{W} = \lambda \mathbf{B} \mathbf{W}, \label{finiteproj}
\EEA
restricting the variational eigenvalue problem on the hypothesis space $\mathcal{W}_N$.}

We first note that $h_{X,M} \to 0$ as $N\to \infty$. For example, if the data are independent and identically uniformly distributed, one can show that {\color{black}$h_{X,M} \leq C_{\text{fill}} \left(\frac{\log N}{N}\right)^{1/d}$ (see Lemma~B.2 in \cite{harlim2023radial}). Since $d=2$, the gap condition becomes $c(\mathbf{A},\mathbf{B})> \frac{3CC_{\text{fill}}\log N}{2N^{1/2}}$, which can always be satisfied for sufficiently large $N$.} Finally, we note that the result above does not allow one to mathematically characterize the gap condition in terms of the eigenvalues of the underlying variational problem \eqref{Bocheigprob} since \eqref{gen-eigvalprob} is a finite-dimensional projection. However, the bias error converges at a faster rate, {\color{black}$h_{X,M}^2 \asymp N^{-1}\log N$}, compared to the manifold learning error, which is of the order {\color{black}$Nh_{X,M}^3 \propto N^{-1/2}\log N$}. This suggests that the gap condition represents how much manifold learning error is allowed relative to the bias error, even when the Crawford number is undefined as $N \to \infty$.

Finally, we note that convergence of the eigenvector fields can be established using the same proof technique as in our previous work \cite{harlim2023radial,peoples2026spectral}. Based on the previous results, we believe that the error of the estimated eigenvector field measured in term of the $L^2$-norm over the point cloud data $X$ is bounded under the same error rate, $h_{X,M}$. In the next section, we will numerically verify this conjecture on several examples. 

\section{Numerical Results}
\label{numerical-results-section}
In this section, we numerically implement the local curved mesh method to estimate the Bochner and Hodge Laplacians on some example manifolds. In Section~\ref{numerical-sphere-subsection}, we numerically demonstrate the convergence of eigenvalues and eigenvectors of the generalized eigenvalue problem in \eqref{gen-eigvalprob} obtained from randomly sampled data to the analytic eigenvalues and eigenvector fields of the Bochner Laplacian on the two-dimensional sphere.In Section~\ref{numerical-torus-subsection}, we numerically demonstrate the convergence of the estimated Hodge Laplacian eigenvalues on a two-dimensional torus.
Finally, since in practical situations, one often has access to noisy observations that do not lie precisely on a manifold, we numerically investigate the robustness of this algorithm to noise in Section \ref{numerical-noise-subsection}. For convenience, we refer to the estimations obtained from the curved mesh method by CMM.
  
\subsection{Bochner and Hodge Laplacian Estimations on Sphere}
\label{numerical-sphere-subsection}

In this section, we demonstrate the numerical convergence of the Hodge and Bochner Laplacians on the two-dimensional unit sphere, $M = \{x \in \mathbb{R}^3 : | x | = 1 \}$, embedded in $\mathbb{R}^3$. To simulate uniform random sampling on this surface, we take the following approach: first, we generate a vector in $\mathbb{R}^3$ via $\vec{w}_i \sim \mathcal{N}(\vec{0}, \mathbf{I})$. We then normalize each sample according to $x_i = \vec{w}_i / | \vec{w}_i |$. An example of the resulting dataset is shown in Figure \ref{data-figure}a.

\begin{figure}[htbp]
{\scriptsize \centering
\begin{tabular}{ccc}
\includegraphics[width=.3\textwidth]{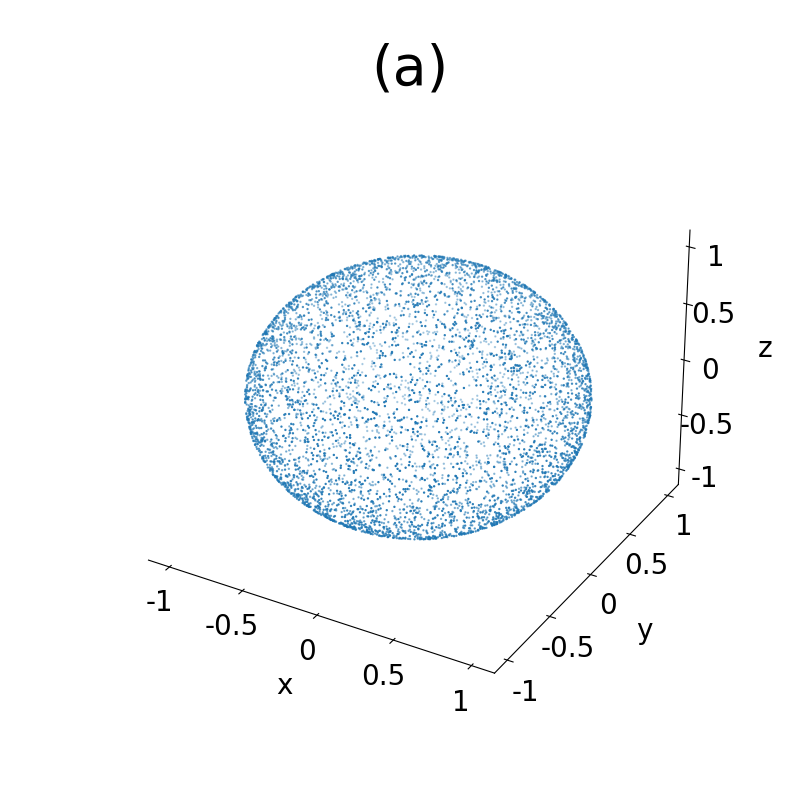} &
\includegraphics[width=.3\textwidth]{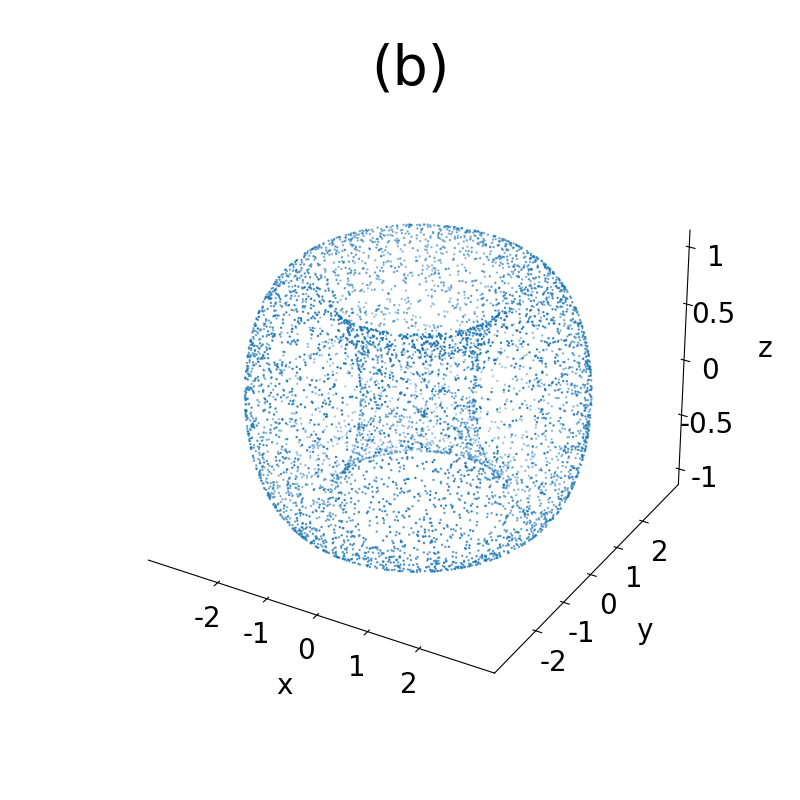} & 
\includegraphics[width=.3\textwidth]{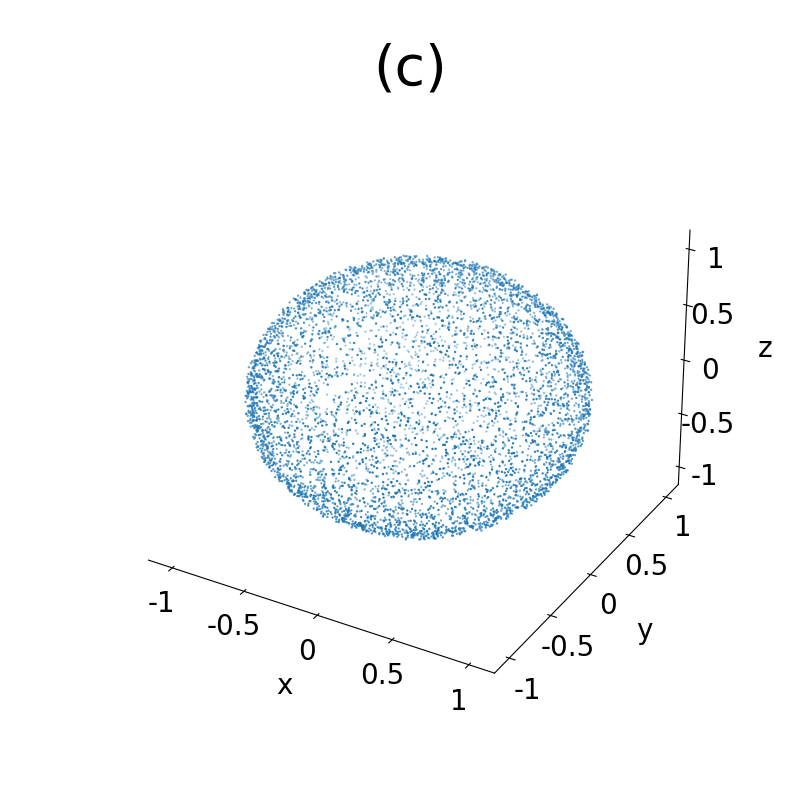}
\end{tabular} }
\caption{Example datasets used for numerical experiments. Displayed is a single trial with $N=4000$ points for (a) the sphere (b) the torus, and (c) the noisy sphere (each data point is perturbed in the normal direction with uniform noise of size $0.10$).}
\label{data-figure}
\end{figure}

Analytical eigensolutions for the Bochner and Hodge Laplacians on this manifold are well known. For a detailed description, see, for example Section 5.4 of \cite{harlim2023radial}. To summarize, the distinct eigenvalues of the Hodge and Bochner Laplacians are given by the following:
\BEA 
\textup{ Hodge: }\quad \lambda_\ell &=& (\ell + 1)\ell, \notag \\
\textup{ Bochner: }\quad \lambda_\ell &=& (\ell + 1)\ell - 1, \qquad \ell \in \mathbb{N}^{+}. \notag 
\EEA 
Since the sphere has constant curvature, the Hodge and Bochner Laplacians share eigenvector fields. Eigenvector fields for the first $3$ distinct eigenspaces are given by 
$\left\{ \mathbf{n} \times \nabla_{\mathbb{R}^3} f, \quad \mathbf{P} \nabla_{\mathbb{R}^3} f \right\},$
where $\mathbf{n}$ denotes the outward facing normal vector, $\mathbf{P}$ denotes the tangential projection tensor projecting $(\nabla_{\mathbb{R}^3} f)(x)$ to $T_xM$, and $f$ is given by one of the spherical harmonics:
\BEA 
f(\vec{v}) &=& (\vec{v})_i, \quad \textup{ for } \ell = 1, \notag \\
f(\vec{v}) &=& 3(\vec{v})_i(\vec{v})_k - \delta_{ik}, \quad \textup{ for } \ell = 2, \notag \\
f(\vec{v}) &=& 15(\vec{v})_i(\vec{v})_j(\vec{v})_k - 3\delta_{ij}(\vec{v})_k - 3\delta_{ki}(\vec{v})_j - 3\delta_{jk}(\vec{v})_i, \quad \textup{ for } \ell = 3, \notag 
\EEA 
where $i,j,k \in \{1,2,3\}$. 

\begin{figure}[htbp]
{\scriptsize \centering
\begin{tabular}{cc}
\includegraphics[width=.45\textwidth]{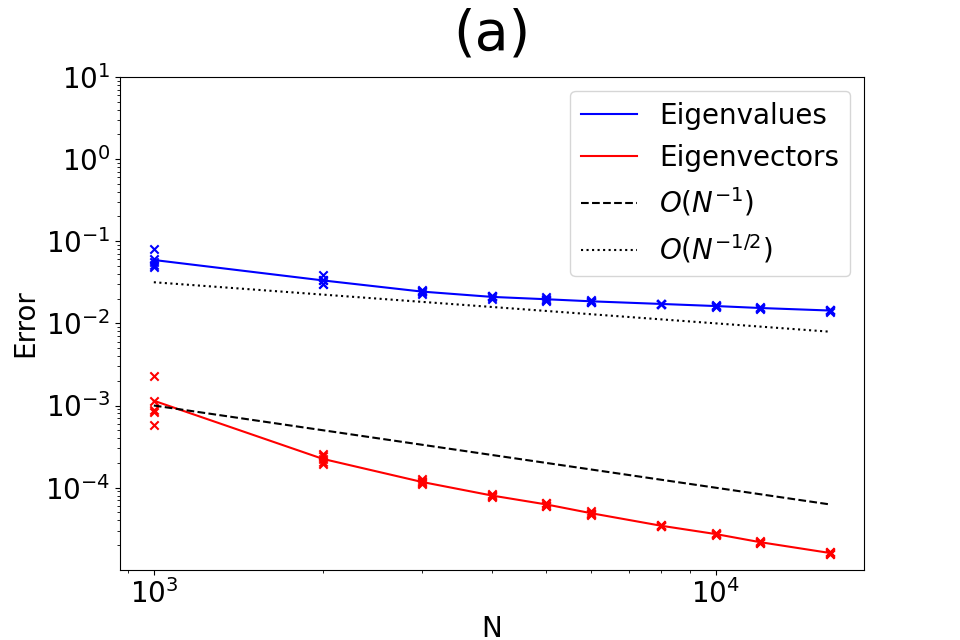} &
\includegraphics[width=.45\textwidth]{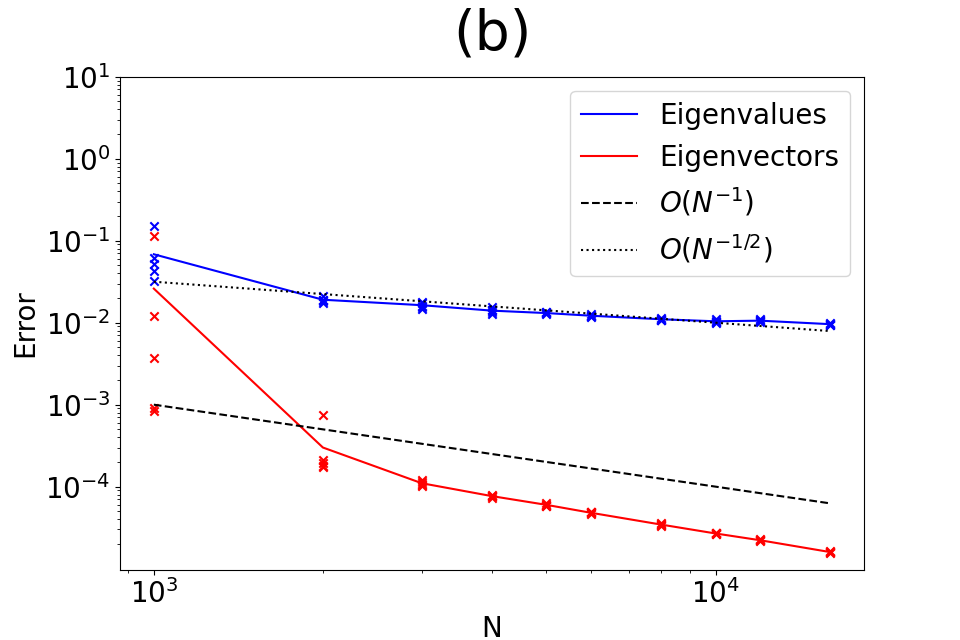}
\end{tabular} }
\caption{Operator estimation on the sphere with uniform sampling. Average relative error (over the leading $L=48$ modes for eigenvalues) and mean-square error (over the leading $L=6$ modes for eigenvectors) as defined in {Equations} \eqref{eigenvalue-metric} {and} \eqref{eigenvector-metric}, respectively: (a) Bochner Laplacian; (b) Hodge Laplacian.
 } 
\label{sphere-convergence-figure} 
\end{figure}

Figure \ref{sphere-convergence-figure} displays the convergence of CMM for the estimation of eigenvalues and eigenvector fields of the Bochner and Hodge Laplacians as a function of $N$. The error metrics used to assess convergence are given by
\BEA 
\textup{Error of eigenvalues} &:=& \frac{1}{L} \sum_{j=1}^L \frac{\left|\lambda_j - \widehat{\lambda}_{j,N}\right|}{\lambda_j} \label{eigenvalue-metric} \\
\textup{Error of eigenvector fields} &:=& \frac{1}{NL} \sum_{i=1}^N \sum_{j=1}^L  \left\| W^{(j)}(x_i) - \mathbf{\widehat{W}}^{(j)}(x_i) \right\|^2,
\label{eigenvector-metric}
\EEA 
where $W^{(j)}(x_i)$ denotes the $j$-th analytic eigenvector field described above, and $\mathbf{\widehat{W}}^{(j)}$ denotes the least-squares linear combination of eigenvectors of the CMM estimator associated with the corresponding estimated eigenspace. We note that the above enumeration of eigenvalues is not distinct (i.e., includes multiplicities). In Figure \ref{sphere-convergence-figure}a, results for estimating the Bochner Laplacian are shown, while in Figure \ref{sphere-convergence-figure}b the corresponding results for the Hodge Laplacian are shown. We note that, in either case, Theorem~\ref{maintheorem} is verified with a convergence rate of $O(N^{-1/2})$ up to a $\log N$ factor. 
In terms of eigenvector fields, the empirical convergence rates agree with our conjectures. For visual comparison, Figure \ref{sphere-visual-comparison-figure} shows estimations (CMM estimation of the Bochner Laplacian) of the first four eigenvector fields of a single trial with $N=6000$. Overall, this example demonstrates that CMM can be used to accurately approximate multiple operators acting on vector fields, with competitive convergence rates. 

\begin{figure}[htbp]
{\scriptsize \centering
\begin{tabular}{ccccc}
\normalsize (a) CMM, mode=1  & \normalsize (b) CMM, mode=2 & \normalsize (c) CMM, mode=3 & \normalsize (d) CMM, mode=4 \\
\includegraphics[width=.2\textwidth]{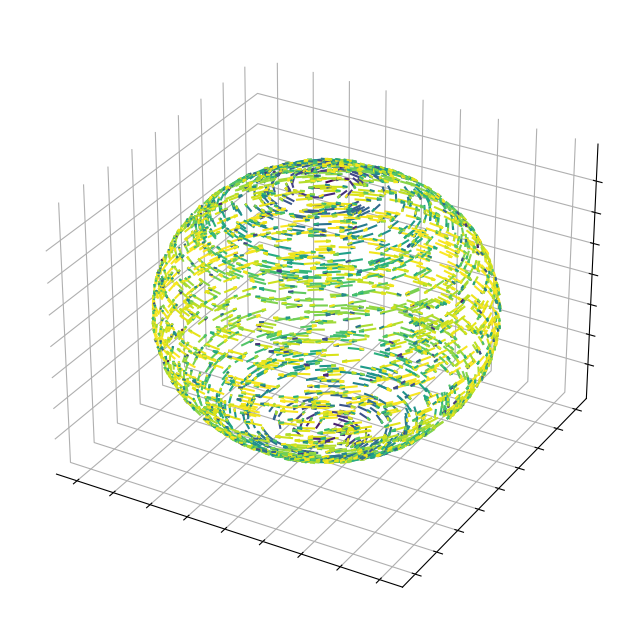} &
\includegraphics[width=.2\textwidth]
{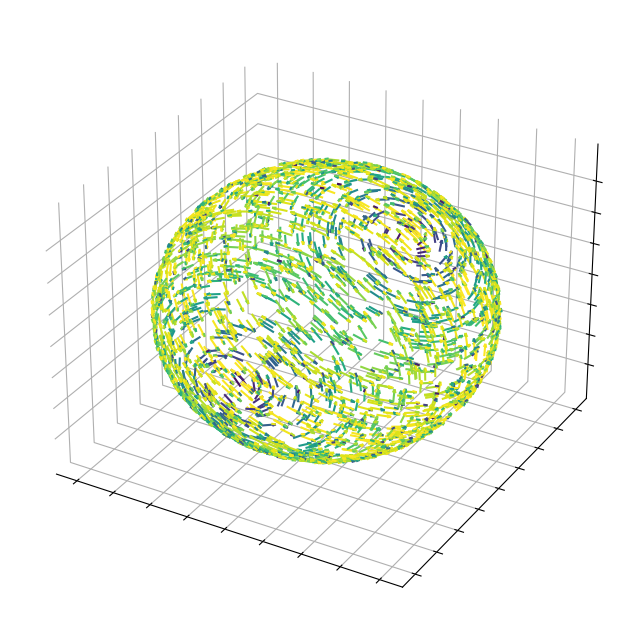} & 
\includegraphics[width=.2\textwidth]{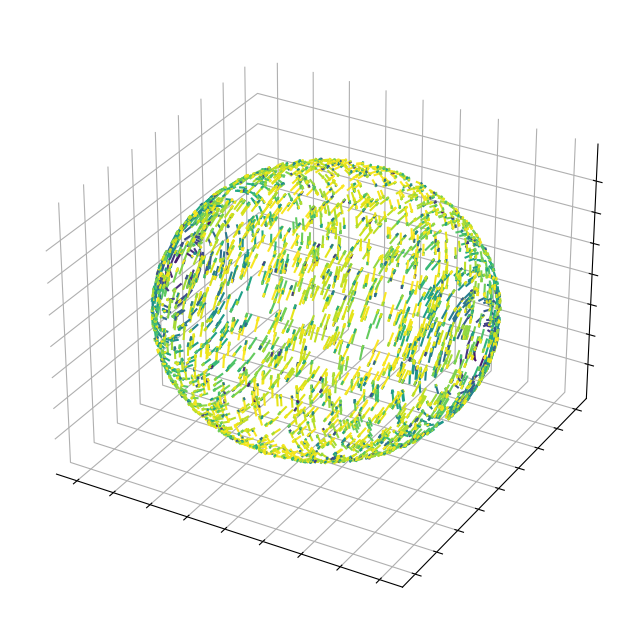} &
\includegraphics[width=.2\textwidth]
{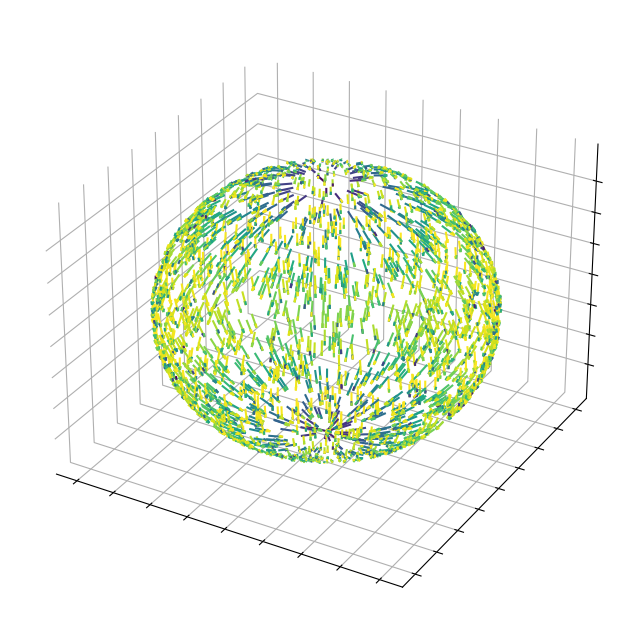} \\
\normalsize (e) truth, mode=1  & \normalsize (f) truth, mode=2 & \normalsize (g) truth, mode=3 & \normalsize (h) truth, mode=4 \\
\includegraphics[width=.2\textwidth]{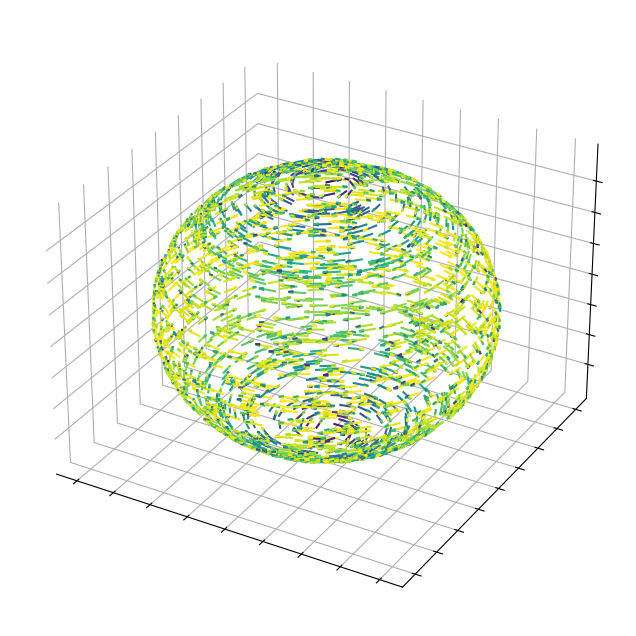} &
\includegraphics[width=.2\textwidth]
{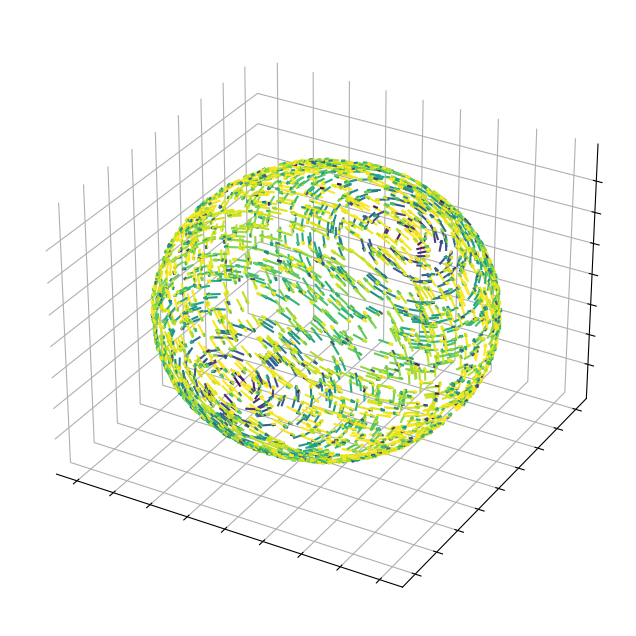} & 
\includegraphics[width=.2\textwidth]{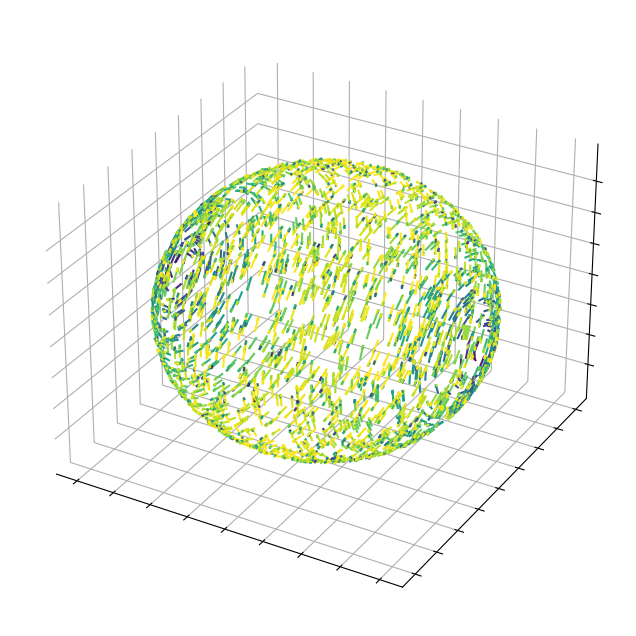} &
\includegraphics[width=.2\textwidth]
{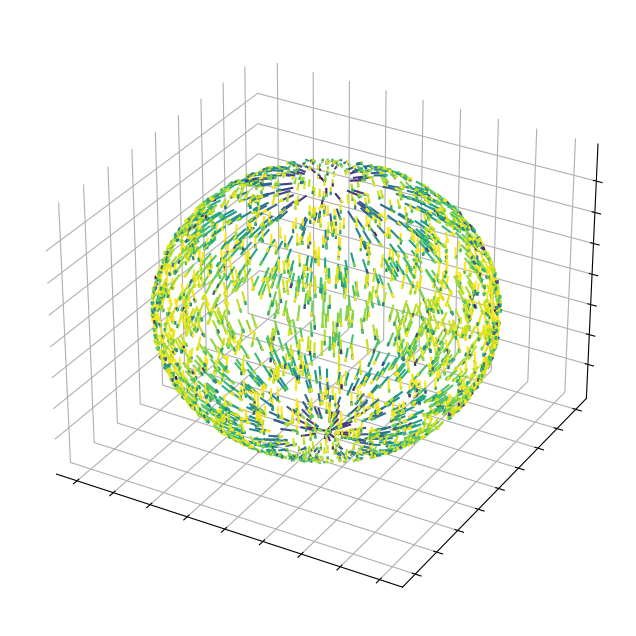} 

\end{tabular} }

\caption{Visual comparison of the first four eigenvector fields for the Bochner Laplacian obtained using CMM (row 1) and analytic methods (row 2). The estimates are obtained from a single trial of $N=6000$ randomly sampled points. Color indicates the vector norm. The vectors shown correspond to a random subset of size $2000$.
 }
\label{sphere-visual-comparison-figure}
\end{figure}

\subsection{Hodge Laplacian Estimation on Torus}
\label{numerical-torus-subsection}
\begin{figure}[htbp]
{\scriptsize \centering
\begin{tabular}{cc}
\includegraphics[width=.45\textwidth]{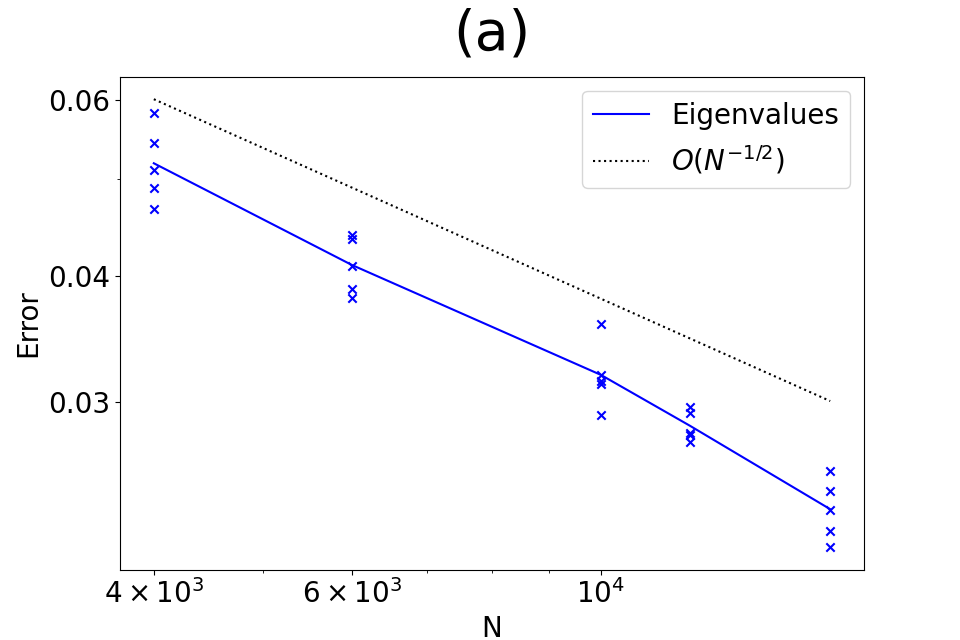} &
\includegraphics[width=.45\textwidth]{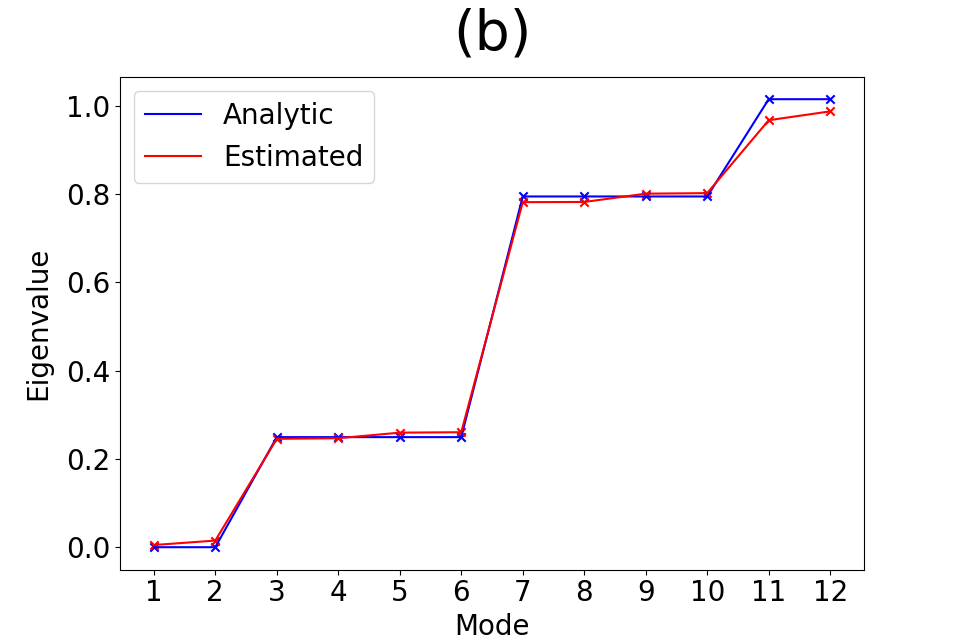} 
\end{tabular} }
\caption{Operator estimation on a torus with uniform sampling (with rejection). (a) Average relative error (over the first $10$ nontrivial eigenvalues) of the Hodge Laplacian estimator obtained from CMM as a function of $N$. (b) Mode-by-mode estimation of the spectrum for a single trial with $N=16000$.
}
\label{torus-convergence-figure}
\end{figure}
In this section, we demonstrate the numerical convergence of CMM in the estimation of the Hodge Laplacian on the $2D$ torus. In particular, consider 
$ 
M = \Bigg\{\left( (2 + \cos \theta ) \cos \phi, (2 + \cos \theta ) \sin \phi, \sin \theta  \right)  : 0 \leq \theta, \phi \leq 2 \pi  \Bigg\}.
$ 
To sample uniformly from the torus, we impose a rejection sampling scheme. That is, we sample a pair $(\theta_i, \phi_i) \sim \textup{Unif}([0, 2\pi]^2)$, along with a rejection parameter $w_i \sim \textup{Unif}([0,1])$. We accept the pair $(\theta_i, \phi_i)$ if  $w_i \leq (2/3 + \cos \theta_i/3)$, in which case we obtain a point $x_i= \left( (2 + \cos \theta_i ) \cos \phi_i, (2 + \cos \theta_i ) \sin \phi_i, \sin \theta_i  \right)$. Otherwise, we reject the sample. A dataset of size $N=$ generated in this way is shown in Figure \ref{data-figure}b.  

We remark that, on the torus, analytic eigensolutions to the eigenvalue problem are not available for the Hodge Laplacian. However, we can leverage the following result, along with semi-analytic eigen-solutions for the Laplace--Beltrami operator, to obtain semi-analytic eigenvalues for the Hodge Laplacian.
\begin{theo}
  Let $M$ denote a two-dimensional manifold embedded in $\mathbb{R}^3$. Then the nontrivial eigenvalues of the Hodge Laplacian are identical to the nontrivial eigenvalues of the Laplace--Beltrami operator, with the multiplicities of the eigenvalues of the Hodge Laplacian being twice those of the Laplace--Beltrami operator.
\end{theo}
The above is a standard result. A detailed proof can be found in \cite{harlim2023radial}. Hence, to evaluate the eigenvalue error, we combine the above result with semi-analytic eigensolutions to the Laplace--Beltrami eigenvalue problem on the torus (following \cite{harlim2023radial}). In particular, the Laplace--Beltrami eigenvalue problem on the two-dimensional torus can be written in intrinsic coordinates as
\BEA 
\Delta_M f_\ell  = - \frac{1}{(2 + \cos \theta)^2}\frac{\partial^2 f_\ell}{\partial \phi^2} - \frac{\partial^2 f_\ell}{\partial \theta^2} + \frac{\sin \theta }{2 + \cos \theta} \frac{\partial f_\ell}{\partial \theta} = \lambda_\ell f_\ell. \notag 
\EEA
Using separation of variables (assuming that $f_\ell(\theta, \phi) = \Phi_\ell (\phi) \Theta_\ell (\theta)$, we reduce the above to the following problems:
\BEA 
\Phi_\ell '' = {\color{black}-m_\ell^2} \Phi_\ell, \qquad {\color{black}-}\Theta_\ell'' {\color{black}+} \frac{\sin \theta}{2 + \cos \theta} \Theta'_\ell {\color{black}+} \frac{m^2_\ell}{(2 + \cos \theta)^2} \Theta_\ell = \lambda_\ell \Theta_\ell. \notag 
\EEA 
The first eigenvalue problem can be solved analytically, while the second eigenvalue problem is solved numerically using a standard finite difference scheme on a fine grid. The resulting values for $\lambda_\ell$ are used to evaluate the error in Figure \ref{torus-convergence-figure}a. Here, we demonstrate the relative error in the estimated Hodge Laplacian eigenvalues obtained using CMM, which numerically verifies the convergence reported in Theorem \ref{maintheorem}. Figure \ref{torus-convergence-figure}b shows the semi-analytic and estimated spectrum mode-by-mode for a single trial with $N=16000$. CMM produces an accurate estimate, especially for modes 0--9.

\subsection{Robustness to Noise}
\label{numerical-noise-subsection}
To investigate the effect of noise, we turn again to the unit sphere example.
Consider the dataset of points $N$ given by $x_i = (1+\epsilon_i) \tilde{x}_i$ where $\tilde{x}_i \sim \textup{Unif}(M)$, and $\epsilon_i \sim \textup{Unif}([-\eta/2, \eta/2])$.  Here, $\textup{Unif}(M)$ denotes uniform sampling with respect to the volume measure of the manifold. To generate such samples, we follow the approach of Section \ref{numerical-sphere-subsection}. The numerical results for the computation of the eigenvalues and eigenvectors for various values of $\eta$ are shown in Figure \ref{noisy-convergence-figure}. 

\begin{figure}[htbp]
{\scriptsize \centering
\begin{tabular}{cc}
\includegraphics[width=.45\textwidth]{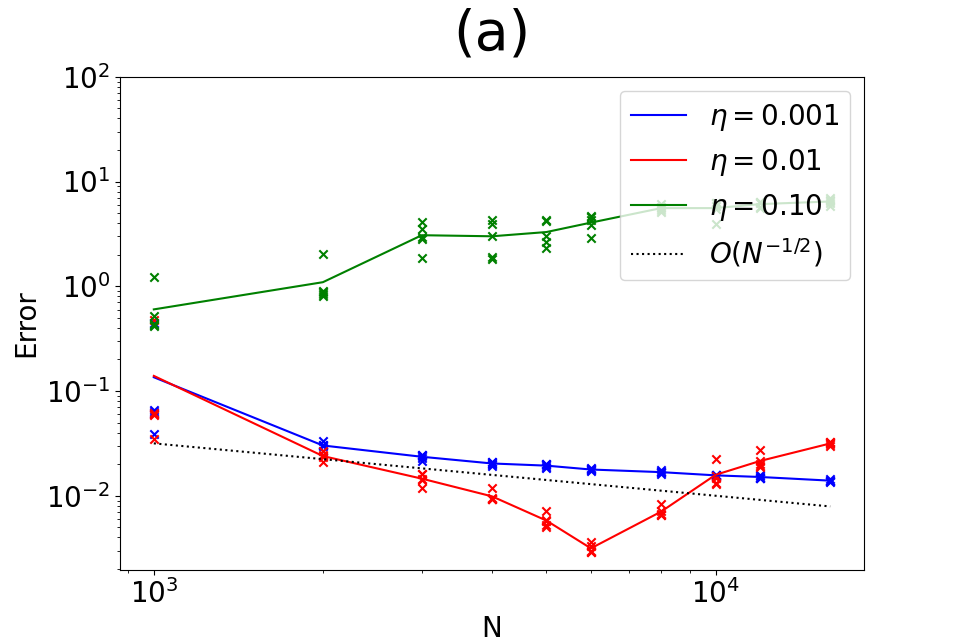} &
\includegraphics[width=.45\textwidth]{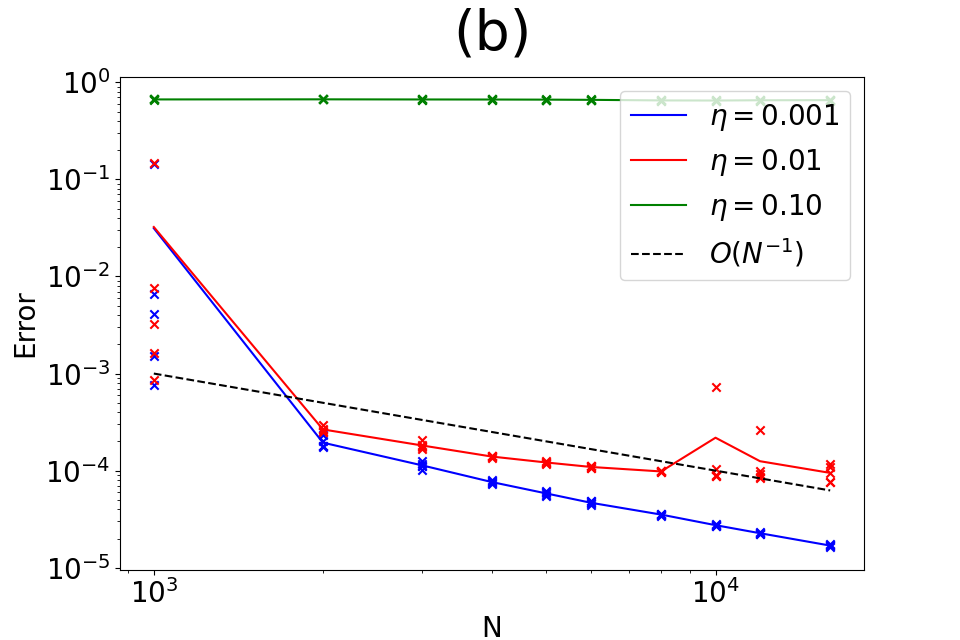}
\end{tabular} }
\caption{Operator estimation on a noisy sphere with uniform sampling distribution as functions of $N$. (a) Average error of eigenvalues. (b) Average mean-square error of eigenvectors. Metrics are as defined in Equations \eqref{eigenvalue-metric} and \eqref{eigenvector-metric} Color indicates the noise level. }
\label{noisy-convergence-figure}
\end{figure}
We see that when the noise is sufficiently small (0.1\% of the radius of the sphere), convergence is still achieved for both the eigenvalues and the eigenvector fields. When the noise increases to 1.0\%, while convergence still occurs for the eigenvector fields (red, Figure \ref{noisy-convergence-figure}b), the convergence of eigenvalues for large $N$ does not occur, though the error in this case is still relatively small. When the noise is increased to 10.0\%, the results for both the eigenvalues and eigenvector fields are inaccurate, and do not converge. Based on these results and the theoretical analysis, we suspect that the algorithm is robust to noise so long as the size of the noise ($\| x_i - \tilde{x_i}\|$ in this example) does not exceed the error induced by this numerical approximation. 

Based on these empirical results, there are many interesting future directions to improve the robustness of CMM to noise. For Example, since CMM relies on an accurate estimation of tangent spaces (Section \ref{local-tangent-space-subsection}), one could employ (or develop) more robust methods for approximating the tangent space on potentially noisy data (see, for instance, \cite{candes2011robust,zhan2011robust}) and characterize the effect of noise in the spectral error analysis.

\section{Summary}
In this paper, we developed a numerical framework for approximating arbitrary differential operators in the weak-form on Riemannian manifolds identified only by randomly sampled point cloud data. We proved, for the case of estimating the Bochner Laplacian, the convergence of eigenvalues, and provided several numerical examples validating this convergence in practice. This framework employs curved (nonlinear) local meshes, which are constructed by projecting point cloud data to the estimated tangent spaces \cite{lai2013local} and using the moving least-squares parameterization \cite{liang2013solving}, thus, allowing the curvature information such as Christoffel symbols to be appropriately accounted for in the operator estimation. This approach is convenient as it avoids the curse of high ambient dimension, the construction of global parameterizations and meshes, and the tuning of various hyperparameters. Moreover, this method is flexible and is not restrictive to estimating only Laplacians. Theoretical and observed convergence rates, which for uniformly sampled data coincide with $O(N^{-1/d})$ for $d = 2$, are competitive with other estimators in the literature in the context of eigenvalue estimation for operators acting on vector fields on unknown manifolds. For example, the spectral convergence of the RBF approximation to the Bochner Laplacian was reported in \cite{harlim2023radial} with convergence rate $O\left(N^{-1/d}\right)$. As for the Graph Laplacian type method, the spectral convergence for Vector Diffusion Maps to the Bochner (or connection) Laplacian was established with no convergence rate as reported in \cite{singer2017spectral}. Instead, they reported pointwise convergence to
the Bochner Laplacian acting on vector fields with a convergence rate $O\left(N^{-\frac{1}{d/2 + 4}}\right)$.

The ideas proposed in this paper suggest many potential directions for future work. First, the theoretical work in this paper assumes that the underlying manifold is without boundary. However, the weak formulation is flexible for complex geometry with arbitrary boundary conditions. Nevertheless, details regarding local parameterization error analysis for points near the boundary remain to be seen. Second, throughout this manuscript, we emphasize solving 1-Laplacian eigenvalue problems. The proposed formulation offers a generic method that can be applied to approximate arbitrary covariant derivatives on general tensor fields for PDE modeling applications. Lastly, as we have shown in the last section, more work needs to be done to overcome noisy data.

\section*{Acknowledgment}

The research of J.H. was partially supported by the NSF grant DMS-2207328 and the ONR grant N00014-26-1-2498. J.W.P was supported as a graduate student under the NSF grant DMS-2207328. The authors thank the anonymous referee for constructive comments that helped improve the presentation.


\appendix
\section{Example Computations for Local Curved Mesh Method} \label{computation-appendix}
Throughout this appendix, local computations are performed in terms of the metric computed in Section \ref{curved-mesh-metric}. For simplicity, we refer to the coordinates of this metric, the entries of this metric, the entries of its inverse, and the Christoffel symbols in this metric as $\{ \partial_{u_1}, \partial_{u_2} \}$, $g_{ij}, g^{ij}$, and $\Gamma^i_{jk}$, respectively. Explicit formulas for all quantities above can be deduced for the metric \eqref{cmm-metric-equation}.  For simplicity, we present the computations for $d=2$. 

The following notation will be neccessary throughout. Fix a block $i$. Let $\mathbf{T} = [\mathbf{t}^{(i)}_1, \mathbf{t}^{(i)}_2, \dots , \mathbf{t}^{(i)}_n]$, the matrix whose columns form an orthonormal basis for the tangent space (first $d$ components) and normal directions (last $n-d$ components) at $i$. Similarly, fix an arbitrary triangle $T_{i,r} \in R(i)$, and let $\mathbf{R}_{i,r} = [\partial_{u_1}, \partial_{u_2}]$ denote the matrix whose columns form a tangent space for the local curved mesh of triangle $T_{i,r}$, as defined in \eqref{cmm-basis-equation}. We remark that $\mathbf{R}_{i,r}$ depends on the point $(u_1,u_2)$. Define  
\BEA  
\label{inner-product}
\mathbf{a}_k^{ij} = \mathbf{R}^{LS}_{i,r} \mathbf{T}^\top \mathbf{t}^{(j)}_k, 
\EEA 
the vector of length $d=2$ corresponding to the representation of $\mathbf{t}^{(j)}_k$ in the local coordinates, where the dependence of $\mathbf{a}^{ij}_k$ on $T_{i,r}$ is understood.
We note that $\mathbf{R}^{LS}_{i,r}= ( \mathbf{R}_{i,r}^\top \mathbf{R}_{i,r})^{-1}\mathbf{R}_{i,r}^\top$ is the least-squares solution for a map that takes the tangent vector $\mathbf{t}^{(j)}_k$ in the coordinates induced by columns of $\mathbf{T}$ to the canonical $\BR^2$ coordinates, which we denoted as $\mathbf{a}_k^{ij}$. With this representation, its clear that 
\BEA
\langle \mathbf{t}^{(i)}_{k_1}, \mathbf{t}^{(j)}_{k_2}   \rangle_{g} = \langle g \mathbf{a}^{ii}_{k_1} , \mathbf{a}^{ij}_{k_2} \rangle = \left(\mathbf{a}^{ii}_{k_1} \right)^\top g \left(\mathbf{a}^{ij}_{k_2}\right), \label{inner_product}
\EEA
as noted in ~\eqref{cmm-vector-field-innerproduct} in Remark~\ref{parallel_transport_remark}, where $g = g_{\Phi_{C_{i,r}}}$ corresponds to the $2\times 2$ metric tensor defined in \eqref{cmm-metric-equation}. 

In practice, since a smooth vector field $W: M \to TM$ can be represented locally on a neighborhood of $T_{x_i}M$, and for sufficiently large $N$ (such that the curvature between neighboring points of data is sufficiently small), $R(i)$ is contained in this neighborhood, instead of $\mathbf{T}$, one can use $\mathbf{T}_i = [\mathbf{t}^{(i)}_1,  \mathbf{t}^{(i)}_2]$, and subsequently ignore the last $(n-2)$ columns of $\mathbf{R}_{i,r}$, which reduces computation time as the first $d=2$ rows of $\mathbf{R}_{i,r}$ do not depend on the point $(u_1,u_2)$; see \eqref{cmm-basis-equation}. We observe the difference in estimated eigenvalues and eigenvector fields when using the full matrices as in \eqref{inner-product} and the reduced approach outlined above is negligible. 

We note that the computations corresponding to off-diagonal matrix elements simplify significantly. In particular, the integrals corresponding to the $(i,j)$-th entry for $i \neq j$ vanish for each $T_{i,r} \in R(i)$ except those for which the edge connecting  $ \vec{0}$ and $ \mathbf{T}_i^\top (x_j-x_i) $ is a side of $T_{i,r}$. Since $T_{i,r}$ is a triangle with vertices $\{\vec{0}, \vec{v}_{i_{j_r}}, \vec{v}_{i_{k_r}}\}$, this simply means that $ \mathbf{T}_i^\top (x_j-x_i)$ is a node of $T_{i,r}$. Therefore, for future convenience, we denote 
\[
Q(i, j) = \{ T_{i,r} \in R(i) : \mathbf{T}_i^\top (x_j-x_i) \textup{ is a node of } T_{i,r} \}.
\]

\subsection{Mass matrix for operators on vector fields}
\label{mass-matrix-computation-appendix}
Here we explicitly compute the mass matrix for vector-fields, $\mathbf{M}$. The diagonal block formula from before reads
$$
\begin{bmatrix} [\mathbf{M}]_{2i-1, 2j -1} & [\mathbf{M}]_{2i-1,2j} \\ [\mathbf{M}]_{2i, 2j-1} & [\mathbf{M}]_{2i,2j} \end{bmatrix} = \sum_{T_{i,r} \in R(i)} \int_{T}  e_i e_j
\begin{bmatrix} 
 \langle \mathbf{t}^{(i)}_{1}, \mathbf{t}^{(j)}_{1}   \rangle_{g}  &  \langle \mathbf{t}^{(i)}_{1}, \mathbf{t}^{(j)}_{2}   \rangle_{g} \\
 \langle \mathbf{t}^{(i)}_{2}, \mathbf{t}^{(j)}_{1}   \rangle_{g} &  \langle \mathbf{t}^{(i)}_{2}, \mathbf{t}^{(j)}_{2}   \rangle_{g}
\end{bmatrix} \sqrt{\det(g)} du_1 du_2,
$$
Since we compute each integral with a quadrature rule on the vertices, it suffices to compute the integrand at an arbitrary point $(u_1$, $u_2)$. 

To begin, we compute the diagonal block components. When $i=j$, we have
\BEA 
\begin{bmatrix} [\mathbf{M}]_{2i-1, 2i -1} & [\mathbf{M}]_{2i-1,2i} \\ [\mathbf{M}]_{2i, 2i-1} & [\mathbf{M}]_{2i,2i} \end{bmatrix} = \sum_{T_{i,r} \in R(i)} \int_{T}  (1 - u_1 - u_2)^2
\begin{bmatrix} 
 \langle  g \mathbf{a}^{ii}_{1}, \mathbf{a}^{ii}_{1}   \rangle  &  \langle g \mathbf{a}^{ii}_{1}, \mathbf{a}^{ii}_{2}   \rangle \\
 \langle g \mathbf{a}^{ii}_{2}, \mathbf{a}^{ii}_{1}   \rangle &  \langle  g \mathbf{a}^{ii}_{2}, \mathbf{a}^{ii}_{2}   \rangle
\end{bmatrix} \sqrt{\det(g)} du_1 du_2, \notag 
\EEA
where we have used \eqref{inner_product}. We remark that the quantities above indeed depend on each triangle $T_{i,r}$. 

When $j \neq i$, we recall that the only two triangles for which the integration is nonzero are those in $Q(i,j)$ (which consists of $0$ or $2$ triangles). In both of these triangles, we can reorder the vertices so that $(1,0) \to \mathbf{T}_i^\top (x_j-x_i)$. Hence, a nonzero of diagonal block of the mass matrix takes the form 
\BEA 
\begin{bmatrix} [\mathbf{M}]_{2i-1, 2j -1} & [\mathbf{M}]_{2i-1,2j} \\ [\mathbf{M}]_{2i, 2j-1} & [\mathbf{M}]_{2i,2j} \end{bmatrix} = \sum_{T_{i,r} \in Q(i,j)} \int_{T}  (1 - u_1 - u_2) u_1
\begin{bmatrix} 
 \langle  g \mathbf{a}^{ii}_{1}, \mathbf{a}^{ij}_{1}   \rangle  &  \langle g \mathbf{a}^{ii}_{1}, \mathbf{a}^{ij}_{2}   \rangle \\
 \langle g \mathbf{a}^{ii}_{2}, \mathbf{a}^{ij}_{1}   \rangle &  \langle  g \mathbf{a}^{ii}_{2}, \mathbf{a}^{ij}_{2}   \rangle
\end{bmatrix} \sqrt{\det(g)} du_1 du_2. \notag 
\EEA

\subsection{Stiffness Matrix for Bochner Laplacian}\label{bochner-laplacian-computation-appendix}

{
Let $g$ be the Riemannian metric on a surface and let
\BEA
\langle a,b\rangle_g=\operatorname{tr}(a\,g\,b^\top g) \label{Riemann_metric}
\EEA
denote the induced inner product on $(2,0)$-tensor fields, written in the basis
$\{\partial_{u_i}\otimes \partial_{u_j}\}_{i,j=1}^2$.

The Bochner Laplacian $\Delta_B$ is defined weakly by
\[
\langle \Delta_B w, v \rangle_g = \langle \textup{grad}_g w, \textup{grad}_g v \rangle_g
\qquad \forall\, w,v\in H^1(TM),
\]
where
\BEA
\textup{grad}_g v
=
g^{kj}\left(\frac{\partial v^i}{\partial u_k}+v^p\Gamma^i_{pk}\right)
\partial_{u_i}\otimes \partial_{u_j}.\label{grad_vector}
\EEA

For the local stiffness matrix, the relevant block is
\[
\begin{bmatrix}
[\mathbf{S}]_{2i-1,2j-1} & [\mathbf{S}]_{2i-1,2j}\\
[\mathbf{S}]_{2i,2j-1} & [\mathbf{S}]_{2i,2j}
\end{bmatrix}
=
\sum_{T_{i,r}\in R(i,j)}
\int_T
\begin{bmatrix}
\langle \textup{grad}_g e_i \mathbf{t}_1^{(i)},\textup{grad}_g e_j \mathbf{t}_1^{(j)}\rangle_g &
\langle \textup{grad}_g e_i \mathbf{t}_1^{(i)},\textup{grad}_g e_j \mathbf{t}_2^{(j)}\rangle_g\\[2mm]
\langle \textup{grad}_g e_i \mathbf{t}_2^{(i)},\textup{grad}_g e_j \mathbf{t}_1^{(j)}\rangle_g &
\langle \textup{grad}_g e_i \mathbf{t}_2^{(i)},\textup{grad}_g e_j \mathbf{t}_2^{(j)}\rangle_g
\end{bmatrix}
\sqrt{\det(g)}\,du_1\,du_2 .
\]

To compute the integrands, write
\[
\mathbf{A}^{ij}_k
=\textup{grad}_g[\phi\mathbf t_k^{(i)} ] =
\textup{grad}_g\!\left[(\mathbf{a}^{(ij)}_{k})_1\,\phi\,\partial_{u_1}
+(\mathbf{a}^{(ij)}_{k})_2\,\phi\,\partial_{u_2}\right] = (\mathbf{a}^{(ij)}_{k})_1\textup{grad}_g[\phi\,\partial_{u_1}] 
+(\mathbf{a}^{(ij)}_{k})_2 \textup{grad}_g[\phi\,\partial_{u_2}],
\]
where $\phi$ depends on the local configuration (indices $i,j$) and $\mathbf{a}^{(ij)}_{k}$ corresponds to the components of $\mathbf{t}_k^{(i)}$ in the basis $\{\partial_{u_1},\partial_{u_2}\}$. 

In the diagonal case $i=j$, one takes
$
\phi = 1-u_1-u_2,
$
and in the off-diagonal case $i\neq j$, after reordering the vertices of the triangle, one takes
$
\phi=u_1.
$

Using the local formula \eqref{grad_vector}, we have
\BEA
\textup{grad}_g [(1 - u_1 - u_2) \partial_{u_1}]  &=& g^{k1} \left(\frac{\partial(1 - u_1 - u_2)}{\partial u_k} + (1 - u_1 - u_2)  \Gamma^1_{1k} \right) \partial_{u_1} \otimes \partial_{u_1} \notag \\
&& + g^{k2} \left(\frac{\partial(1 - u_1 - u_2)}{\partial u_k} + (1 - u_1 - u_2)  \Gamma^1_{1k} \right) \partial_{u_1} \otimes \partial_{u_2} \notag \\
&& + g^{k1} \left((1 - u_1 - u_2)  \Gamma^2_{1k} \right) \partial_{u_2} \otimes \partial_{u_1}  + g^{k2} \left((1 - u_1 - u_2)  \Gamma^2_{1k} \right) \partial_{u_2} \otimes \partial_{u_2} \notag
 \\
&=& \left( -g^{11} - g^{21} + (1 - u_1 - u_2)  (\Gamma^{1}_{11} g^{11} +\Gamma^{1}_{12} g^{21}) \right) \partial_{u_1} \otimes \partial_{u_1} \notag \\ 
&& + \left( -g^{12} - g^{22} + (1 - u_1 - u_2)  (\Gamma^{1}_{11} g^{12} +\Gamma^{1}_{12} g^{22}) \right) \partial_{u_1} \otimes \partial_{u_2} \notag \\
&& + \left( (1 - u_1 - u_2)  (\Gamma^{2}_{11} g^{11} +\Gamma^{2}_{12} g^{21}) \right) \partial_{u_2} \otimes \partial_{u_1} \notag \\
&& + \left( (1 - u_1 - u_2)  (\Gamma^{2}_{11} g^{12} +\Gamma^{2}_{12} g^{22}) \right) \partial_{u_2} \otimes \partial_{u_2}, \notag
\EEA
which can be written in matrix form as follows:
\BEA 
\textup{grad}_g [(1 - u_1 - u_2) \partial_{u_1}]  = \left( (1-u_1 - u_2) \begin{bmatrix} 
\Gamma^1_{11} & \Gamma^1_{12} \\
\Gamma^2_{11} & \Gamma^2_{12}
\end{bmatrix} 
+ 
\begin{bmatrix} 
-1 & -1 \\
0 & 0
\end{bmatrix} \right) g^{-1}. \notag 
\EEA
Following the same calculation, we have:
\BEA 
\textup{grad}_g [(1 - u_1 - u_2) \partial_{u_2}]  &=& \left( (1-u_1 - u_2) \begin{bmatrix} 
\Gamma^1_{21} & \Gamma^1_{22} \\
\Gamma^2_{21} & \Gamma^2_{22}
\end{bmatrix} 
+ 
\begin{bmatrix} 
0 & 0 \\
-1 & -1
\end{bmatrix} \right) g^{-1} \notag \\
\textup{grad}_g [ u_1 \partial_{u_1}] &=&  
\left( u_1 
\begin{bmatrix}
  \Gamma^1_{11} & \Gamma^1_{12} \\
  \Gamma^2_{11} & \Gamma^2_{12}  
\end{bmatrix} + 
\begin{bmatrix}
    1 & 0 \\
    0 & 0
\end{bmatrix}
\right) g^{-1}
\notag \\
\textup{grad}_g [ u_1 \partial_{u_2}] &=& \left( u_1 
\begin{bmatrix}
  \Gamma^1_{21} & \Gamma^1_{22} \\
  \Gamma^2_{21} & \Gamma^2_{22}  
\end{bmatrix} + 
\begin{bmatrix}
    0 & 0 \\
    1 & 0
\end{bmatrix}
\right) g^{-1}.  \notag
\EEA 

By the Riemannian metric formula in \eqref{Riemann_metric}, we have:
\[
\langle \textup{grad}_g e_i \mathbf{t}_{k_1}^{(i)},\textup{grad}_g e_j \mathbf{t}_{k_2}^{(j)}\rangle_g
=
\operatorname{tr}\!\left(\mathbf{A}^{ij}_{k_1}\,g\,(\mathbf{A}^{ij}_{k_2})^\top g\right).
\]

Hence, for the diagonal block $i=j$,
In summary, when $i = j$, 
\BEA
\begin{bmatrix} [\mathbf{S}]_{2i-1, 2i -1} & [\mathbf{S}]_{2i-1,2i} \\ [\mathbf{S}]_{2i, 2i-1} & [\mathbf{S}]_{2i,2i} \end{bmatrix} =
\sum_{T_{i,r} \in R(i)} \int_{T}  
\begin{bmatrix} 
\textup{trace} \left(  \mathbf{A}^{ii}_{1} g (\mathbf{A}^{ii}_{1})^\top g  \right) &  \textup{trace} \left(  \mathbf{A}^{ii}_{1} g (\mathbf{A}^{ii}_{2})^\top g  \right) \\
 \textup{trace} \left(  \mathbf{A}^{ii}_{2} g (\mathbf{A}^{ii}_{1})^\top g  \right) &  \textup{trace} \left(  \mathbf{A}^{ii}_{2} g (\mathbf{A}^{ii}_{2})^\top g  \right)
\end{bmatrix} \sqrt{\det(g)} du_1 du_2,\notag
\EEA
while for the off-diagonal block $i \neq j$, we have
\BEA
\begin{bmatrix} [\mathbf{S}]_{2i-1, 2j -1} & [\mathbf{S}]_{2i-1,2j} \\ [\mathbf{S}]_{2i, 2j-1} & [\mathbf{S}]_{2i,2j} \end{bmatrix} =
\sum_{T_{i,r} \in Q(i,j)} \int_{T}  
\begin{bmatrix} 
\textup{trace} \left(  \mathbf{A}^{ii}_{1} g (\mathbf{A}^{ij}_{1})^\top g  \right) &  \textup{trace} \left(  \mathbf{A}^{ii}_{1} g (\mathbf{A}^{ij}_{2})^\top g  \right) \\
 \textup{trace} \left(  \mathbf{A}^{ii}_{2} g (\mathbf{A}^{ij}_{1})^\top g  \right) &  \textup{trace} \left(  \mathbf{A}^{ii}_{2} g (\mathbf{A}^{ij}_{2})^\top g  \right)
\end{bmatrix} \sqrt{\det(g)} du_1 du_2,\notag
\EEA

This yields the local integrands required for the mesh-based assembly of the stiffness matrix.
}

\subsection{Stiffness Matrix for Hodge Laplacian}
\label{hodge-laplacian-computation-appedix}
The Hodge Laplacian is typically defined on a $k$-form $\omega$ by the formula 
$$
\Delta_H \omega = d^\star d \omega + d d^\star \omega,
$$
where $d$ denotes the total differential, and $d^\star$ is the codifferential. The Hodge Laplacian can also be defined on vector fields by leveraging the musical isomorphisms:
$$
\Delta_H v := \sharp (d^\star d \omega + d d^\star) \flat v, 
$$

For the purposes of this paper, only the weak-form is necessary:
\BEA 
\label{weak-hodge}
\langle \Delta_H v, w \rangle_g = \langle d \flat v, d \flat w \rangle_g + \langle d^\star \flat v, d^\star \flat w \rangle_g,
\EEA
for any $v, w \in \mathfrak{X}(M)$, so the first inner-product is defined over $2$-forms (or $(0,2)$-tensors) whereas the second inner product is defined over $0$-forms (or functions). Based on the above formula, all that is needed for computations is the action of $d$ and $d^\star$ on $1$-forms (on manifolds of dimension $2$). These formulas are given by
\BEA 
d \left( \omega_i d u^i \right) &=& \frac{\partial \omega_i}{\partial u_j} d u_j, \notag \\
d^\star \left( \omega_i d u^i \right) &=& (-1) \star d \star \left( \omega_i d u^i \right), \notag 
\EEA
where $\star$ denotes the hodge star operator, which is defined locally on a $k$-form $\omega$ to be the unique $(d-k)$ form $\star \omega$ satisfying the formula
\BEA 
\omega' \wedge \star \omega = \langle \omega' , \omega \rangle_g \sqrt{ \det g } d u_1 \wedge d u_2  \notag 
\EEA 
for any $k$-form $\omega'$. Since \eqref{weak-hodge}, as well as the above formula, contains inner-products of $k$-forms, we remind the reader that the inner product of two $k$-forms is given by
\BEA 
\langle \omega^1 \wedge \dots \wedge \omega^k , \eta^1 \wedge \dots \wedge \eta^k \rangle = \det ( \langle \omega^i, \eta^j \rangle_g ). \notag 
\EEA

The definition for the stiffness matrix can be initially simplified as follows:
\BEA
\label{hodge-stiffness}
&& \begin{bmatrix} [\mathbf{S}]_{2i-1, 2j -1} & [\mathbf{S}]_{2i-1,2j} \\ [\mathbf{S}]_{2i, 2j-1} & [\mathbf{S}]_{2i,2j} \end{bmatrix} =  \\
&& \sum_{T_{i,r} \in R(i)} \int_{T}  
\begin{bmatrix} 
 \langle d\flat e_i \mathbf{t}^{(i)}_{1}, d \flat e_j \mathbf{t}^{(j)}_{1}   \rangle_{g} + \langle d^\star \flat e_i \mathbf{t}^{(i)}_{1}, d^\star \flat e_j \mathbf{t}^{(j)}_{1}   \rangle_{g}  &   \langle d\flat e_i \mathbf{t}^{(i)}_{1}, d \flat e_j \mathbf{t}^{(j)}_{2}   \rangle_{g} + \langle d^\star \flat e_i \mathbf{t}^{(i)}_{1}, d^\star \flat e_j \mathbf{t}^{(j)}_{2}   \rangle_{g} \\
\langle d\flat e_i \mathbf{t}^{(i)}_{2}, d \flat e_j \mathbf{t}^{(j)}_{1}   \rangle_{g} + \langle d^\star \flat e_i \mathbf{t}^{(i)}_{2}, d^\star \flat e_j \mathbf{t}^{(j)}_{1}   \rangle_{g} &   \langle d\flat e_i \mathbf{t}^{(i)}_{2}, d \flat e_j \mathbf{t}^{(j)}_{2}   \rangle_{g} + \langle d^\star \flat e_i \mathbf{t}^{(i)}_{2}, d^\star \flat e_j \mathbf{t}^{(j)}_{2}   \rangle_{g}
\end{bmatrix} \sqrt{\det(g)} du_1 du_2.\notag
\EEA
We'll begin by computing $d \flat v$ and $d^\star \flat v$ on basis functions in each case. Note that in general, the nodal basis functions are a linear combination of 
$$
v = c_1 \phi \frac{\partial}{\partial u_1} + c_2 \phi \frac{\partial }{\partial u_2},
$$
where $\phi = 1 - u_1 - u_2$ when $i\neq j$ and $\phi=u_1$ when $i=j$.

Hence, we handle the following four cases:
\BEA 
\textup{case 1:}&& \quad v - (1-u_1 - u_2) \frac{\partial}{\partial u_1} \notag \\
\textup{case 2:}&& \quad v = (1-u_1 - u_2) \frac{\partial}{\partial u_2} \notag \\
\textup{case 3:}&& \quad v = u_1\frac{\partial}{\partial u_1} \notag \\
\textup{case 4:}&& \quad v = u_1\frac{\partial}{\partial u_2}. \notag 
\EEA

For case 1, we have $\flat v = (1-u_1 - u_2) g_{11} du_1 + (1-u_1 - u_2) g_{12} du_2$. Hence, 
\BEA
d \flat v &=& \frac{\partial (g_{11} \hat{\phi})}{\partial u_2} du_2 \wedge du_1 + \frac{\partial ( g_{12} \hat{\phi})}{\partial u_1} du_1 \wedge du_2 = \left( \frac{\partial ( g_{12} \hat{\phi})}{\partial u_1} - \frac{\partial (g_{11} \hat{\phi})}{\partial u_2} \right) du_1 \wedge du_2 \notag \\
&=& \left( g_{11} - g_{12} + \hat{\phi} \frac{\partial g_{12}}{\partial u_1} -  \hat{\phi}\frac{\partial g_{11}}{\partial u_2} \right) du_1 \wedge du_2 := A_{1} du_1 \wedge du_2.  \notag
\EEA
Similarly, we have that
\BEA 
d^\star \flat v = (-1) \star d \star \left( \hat{\phi} g_{11} du_1 + \hat{\phi} g_{12} du_2 \right).
\notag 
\EEA 
The formula for $\star$ shows that $\star (du_1) = -g^{12} \sqrt{\det(g)} du_1 + g^{11} \sqrt{\det(g)} du_2$, $\star(du_2) = -g^{22} \sqrt{\det(g)} du_1 + g^{12} \sqrt{\det(g)} du_2$, and $\star(du_1 \wedge du_2) = \det(g^{-1}) \sqrt{\det(g)} = (\det(g))^{-1/2}$. Hence, 
\BEA 
d^\star \flat v &=& (-1) \star d  \left( \hat{\phi}\sqrt{\det(g)} \left( (-g_{11} g^{12} - g_{12}g^{22})du_1   + (g_{12}g^{12} + g_{11}g^{11})du_2\right) \right) \notag \\
&=& (-1) \star d  \left( \hat{\phi}\sqrt{\det(g)} \left(1 du_2\right) \right) \notag \\
&=& (-1) \star \left( \frac{\partial}{\partial u_1} \left( \hat{\phi}\sqrt{\det(g)} \right) du_1 \wedge du_2 \right) 
\notag \\
&=& \frac{-1}{\sqrt{\det(g)}} \left( \frac{\partial}{\partial u_1} \left( \hat{\phi}\sqrt{\det(g)} \right) \right) \notag \\
&=& 1 - \frac{\hat{\phi}}{\sqrt{\det g}} \frac{\partial \sqrt{\det g}}{\partial u_1} := A_3. \notag
\EEA
Similar computations yield the results for other three cases. The results are as follows:
\BEA 
\textup{ case 2:} && \quad d \flat v = \left( g_{12} - g_{22} + \hat{\phi} \frac{\partial g_{22} }{\partial u_1} - \hat{\phi} \frac{\partial g_{12}}{\partial u_2} \right) du_1 \wedge du_2 :=  A_2 du_1 \wedge du_2 \notag \\
&& \quad d^\star \flat v = 1 - \frac{\hat{\phi}}{\sqrt{\det g}} \frac{\partial \sqrt{\det g}}{\partial u_2} := A_4
\notag \\
\textup{ case 3:} && \quad d \flat v = \left( g_{21} + u_1 \frac{\partial g_{21}}{\partial u_1} - u_1 \frac{\partial g_{11}}{\partial u_2} \right) du_1 \wedge du_2 := B_{1} du_1 \wedge du_2  \notag \\
&& \quad d^\star \flat v = -1 - \frac{u_1}{\sqrt{\det g}} \frac{\partial \sqrt{\det g}}{\partial u_1} := B_3 \notag \\
\textup{ case 4:} && \quad d\flat v = \left( g_{22} + u_1 \frac{\partial g_{22}}{\partial u_1} - u_1 \frac{\partial g_{12}}{\partial u_2}\right) du_1 \wedge du_2 := B_2 du_1 \wedge du_2 \notag \\
&& \quad d^\star \flat v = - \frac{u_1}{\sqrt{\det g}} \frac{\partial \sqrt{\det g}}{\partial u_2} := B_4. \notag 
\EEA 

We can now compute the two terms in each of the two cases: the diagonal  and the off diagonal blocks. For the diagonal blocks ($i=j$), we have 
\BEA
\langle d \flat e_i \mathbf{t}^{(i)}_{r},  d \flat e_i \mathbf{t}^{(i)}_{s} \rangle_g &=& (\mathbf{a}^{ii}_r)_1 (\mathbf{a}^{ii}_s)_1 A^2_{1} \langle du_1 \wedge du_2, du_1 \wedge du_2 \rangle_g + (\mathbf{a}^{ii}_r)_2  (\mathbf{a}^{ii}_s)_1 A_1 A_2 \langle du_1 \wedge du_2, du_1 \wedge du_2 \rangle_g \notag \\
&+& (\mathbf{a}^{ii}_r)_1  (\mathbf{a}^{ii}_s)_2 A_1 A_2 \langle du_1 \wedge du_2, du_1 \wedge du_2 \rangle_g + (\mathbf{a}^{ii}_r)_2  (\mathbf{a}^{ii}_s)_2 A_2^2 \langle du_1 \wedge du_2, du_1 \wedge du_2 \rangle_g \notag \\
&=& \frac{1}{\det g} \left( (\mathbf{a}^{ii}_r)_1  (\mathbf{a}^{ii}_s)_1 A^2_{1} + (\mathbf{a}^{ii}_r)_2  (\mathbf{a}^{ii}_s)_1 A_1 A_2 + (\mathbf{a}^{ii}_r)_1  (\mathbf{a}^{ii}_s)_2 A_1 A_2 + (\mathbf{a}^{ii}_r)_2  (\mathbf{a}^{ii}_s)_2 A_2^2 \right), \label{hodge-diag-1}
\EEA
and
\BEA
\langle d^\star \flat e_i \mathbf{t}^{(i)}_{r},  d^\star \flat e_i \mathbf{t}^{(i)}_{s} \rangle_g =  \left(  (\mathbf{a}^{ii}_r)_1 (\mathbf{a}^{ii}_s)_1 A^2_{3} +  (\mathbf{a}^{ii}_r)_2 (\mathbf{a}^{ii}_s)_1 A_3 A_4 +  (\mathbf{a}^{ii}_r)_1 (\mathbf{a}^{ii}_s)_2 A_3 A_4 +  (\mathbf{a}^{ii}_r)_2  (\mathbf{a}^{ii}_s)_2 A_4^2 \right). \label{hodge-diag-2}
\EEA
For the off diagonal blocks ($i \neq j$), we have
\BEA
\langle d \flat e_i \mathbf{t}^{(i)}_{r},  d \flat e_i \mathbf{t}^{(j)}_{s} \rangle_g &=& (\mathbf{a}^{ii}_r)_1 (\mathbf{a}^{ij}_s)_1 A_1 B_1 \langle du_1 \wedge du_2, du_1 \wedge du_2 \rangle_g  + (\mathbf{a}^{ii}_r)_1 (\mathbf{a}^{ij}_s)_2 A_1 B_2 \langle du_1 \wedge du_2, du_1 \wedge du_2 \rangle_g \notag \\
&+& (\mathbf{a}^{ii}_r)_2 (\mathbf{a}^{ij}_s)_1 A_2 B_1 \langle du_1 \wedge du_2, du_1 \wedge du_2 \rangle_g + (\mathbf{a}^{ii}_r)_2 (\mathbf{a}^{ij}_s)_2 A_2 B_2 \langle du_1 \wedge du_2, du_1 \wedge du_2 \rangle_g  \notag \\
&=&  \frac{1}{\det g} \left( (\mathbf{a}^{ii}_r)_1 (\mathbf{a}^{ij}_s)_1 A_1 B_1 + (\mathbf{a}^{ii}_r)_1 (\mathbf{a}^{ij}_s)_2 A_1 B_2 + (\mathbf{a}^{ii}_r)_2 (\mathbf{a}^{ij}_s)_1 A_2 B_1 + (\mathbf{a}^{ii}_r)_2 (\mathbf{a}^{ij}_s)_2 A_2 B_2 \right), \label{hodge-offdiag-1}
\EEA
and, 
\BEA
\langle d^\star \flat e_i \mathbf{t}^{(i)}_{r},  d^\star \flat e_i \mathbf{t}^{(j)}_{s} \rangle_g =  \left( (\mathbf{a}^{ii}_r)_1 (\mathbf{a}^{ij}_s)_1 A_{3}B_{3} + (\mathbf{a}^{ii}_r)_2 (\mathbf{a}^{ij}_s)_1 B_3 A_4 + (\mathbf{a}^{ii}_r)_1 (\mathbf{a}^{ij}_s)_2 A_3 B_4 + (\mathbf{a}^{ii}_r)_2 (\mathbf{a}^{ij}_s)_2 A_4 B_4\right). \label{hodge-offdiag-2}
\EEA
Inserting \eqref{hodge-diag-1}, \eqref{hodge-diag-2} into \eqref{hodge-stiffness}, we obtain the formula for the diagonal elements. The off diagonal elements are obtained by inserting \eqref{hodge-offdiag-1}, \eqref{hodge-offdiag-2} into \eqref{hodge-stiffness}.

\section{Proofs of Technical Lemmas}
\label{proofs-of-technical-lemmas}

\begin{proof}[Proof of Lemma \ref{local-metric-approximation-error} ]

For simplicity of notation, we do not explicitly write the evaluation of each quantity on $\vec{v}$. First, notice that
\BEA 
\partial_{v_1}(\gamma_{x_i}) = \Bigg(1,0, \frac{\partial z_{x_i}}{\partial v_1} \Bigg)^\top, \qquad
\partial_{v_2}(\gamma_{x_i}) = \Bigg(0,1, \frac{\partial z_{x_i}}{\partial v_2} \Bigg)^\top.\notag 
\EEA
Similarly, 
\BEA 
\partial_{v_1}(\alpha_{x_i}) =  \Bigg(1,0, \frac{\partial p_i}{\partial v_1} \Bigg)^\top, \qquad
\partial_{v_2}(\alpha_{x_i}) =  \Bigg(0,1, \frac{\partial p_i}{\partial v_2} \Bigg)^\top.  \notag 
\EEA
Using \eqref{GMLSerrorbdd}, we see that 
\BEA 
  \left\| \partial_{v_k} \left(  \gamma_{x_i} \right)(\vec{v}) - \partial_{v_k} \left(  \alpha_{x_i} \right)(\vec{v})  \right\| &=& O(h_{X, M}^2). \notag 
\EEA 
Moreover, adding and subtracting $(\partial p_i / \partial v_q) (\partial z_{x_i}  / \partial v_s)$, we find that 
    \BEA 
    |(g_{\alpha_{x_i}})_{qs} - (g_{\gamma_{x_i}})_{qs}| &\leq&  
    |(\partial p_i / \partial v_q) (\partial p_i  / \partial v_s) - (\partial p_i / \partial v_q) (\partial z_{x_i}  / \partial v_s) | \notag \\  
    &&+  |(\partial p_i / \partial v_q) (\partial z_{x_i}  / \partial v_s) - (\partial z_{x_i} / \partial v_q) (\partial z_{x_i}  / \partial v_s)| = O(h_{X, M}^2). \notag 
    \EEA 
    Similarly, adding and subtracting $(g_{\alpha_{x_i}})_{11} (g_{\gamma_{x_i}})_{22} + (g_{\alpha_{x_i}})_{12} (g_{\gamma_{x_i}})_{21}$, we have that  
    \BEA 
     | \det(g_{\gamma_{x_i}}) - \det(g_{\alpha_{x_i}})  | &\leq& \Bigg| (g_{\gamma_{x_i}})_{11} (g_{\gamma_{x_i}})_{22} - (g_{\alpha_{x_i}})_{11} (g_{\gamma_{x_i}})_{22} \Bigg| \notag \\
     &+& \Bigg| (g_{\alpha_{x_i}})_{11} (g_{\gamma_{x_i}})_{22}  - (g_{\alpha_{x_i}})_{11} (g_{\alpha_{x_i}})_{22}  \Bigg| \notag \\
     &+& \Bigg| (g_{\gamma_{x_i}})_{12} (g_{\gamma_{x_i}})_{21} - (g_{\alpha_{x_i}})_{12} (g_{\gamma_{x_i}})_{21} \Bigg| \notag \\
     &+& \Bigg| (g_{\alpha_{x_i}})_{12} (g_{\gamma_{x_i}})_{21}  - (g_{\alpha_{x_i}})_{12} (g_{\alpha_{x_i}})_{21}  \Bigg| \notag \\
     &=& O(h_{X,M}^2).  \notag 
    \EEA 
    Since $\det(g_{\gamma_{x_i}})$ is bounded away from $0$ independent of $x_i$, we have that   
    \BEA 
     \left| \sqrt{\det(g_{\gamma_{x_i}})} - \sqrt{\det(g_{\alpha_{x_i}})}  \right| {\color{black}\leq \frac{\left|\det(g_{\gamma_{x_i}}) - \det(g_{\alpha_{x_i}})  \right|}{\left| \sqrt{\det(g_{\gamma_{x_i}})} + \sqrt{\det(g_{\alpha_{x_i}})}  \right|}= O(h_{X,M}^2).} \notag 
    \EEA
    Using the above facts along with the entrywise formula for the coefficients of the inverse (see, for instance, Chapter 5, Theorem 7 of \cite{lax2007linear}), its immediate to see that,
    \BEA 
       \left| g^{qs}_{\gamma_{x_i}} - g^{qs}_{p_i}\right| = {\color{black}O(h_{X,M}^2)}.  \notag 
    \EEA 
    For the Christoffel symbols, recall that $p_i$ can estimate two derivatives of $z_{x_i}$ up to order $h_{X,M}$. Since 
    \BEA 
    \frac{\partial (g_{\gamma_{x_i}})_{qs}}{ \partial v_t} &=&  (\partial^2 z_{x_i} / \partial v_q v_t) (\partial z_{x_i}  / \partial v_s) +  (\partial z_{x_i} / \partial v_q) (\partial^2 z_{x_i}  / \partial v_s v_t) \notag \\
    \frac{ \partial (g_{\alpha_{x_i}})_{qs} } { \partial v_t } &=&  (\partial^2 p_{i} / \partial v_q v_t) (\partial p_{i}  / \partial v_s) +  (\partial p_{i} / \partial v_q) (\partial^2 p_{i}  / \partial v_s v_t), \notag 
    \EEA
    adding and subtracting $ (\partial^2 z_{x_i} / \partial v_q v_t) (\partial p_{i}  / \partial v_s) +  (\partial z_{x_i} / \partial v_q) (\partial^2 p_{i}  / \partial v_s v_t) $, we see that 
    \BEA 
    \Bigg| \partial (g_{\gamma_{x_i}})_{qs} / \partial v_t -  \partial (g_{\alpha_{x_i}})_{qs} / \partial v_t \Bigg| &\leq& \Bigg| (\partial^2 z_{x_i} / \partial v_q v_t) (\partial z_{x_i}  / \partial v_s) - (\partial^2 z_{x_i} / \partial v_q v_t) (\partial p_{i}  / \partial v_s)  \Bigg| \notag \\
    &+& \Bigg| (\partial^2 z_{x_i} / \partial v_q v_t) (\partial p_{i}  / \partial v_s) - (\partial^2 p_{i} / \partial v_q v_t) (\partial p_{i}  / \partial v_s) \Bigg| \notag \\
    &+& \Bigg|(\partial z_{x_i} / \partial v_q) (\partial^2 z_{x_i}  / \partial v_s v_t) -  (\partial z_{x_i} / \partial v_q) (\partial^2 p_{i}  / \partial v_s v_t) \Bigg| \notag \\
    &+& \Bigg|(\partial z_{x_i} / \partial v_q) (\partial^2 p_{i}  / \partial v_s v_t) - (\partial p_{i} / \partial v_q) (\partial^2 p_{i}  / \partial v_s v_t) \Bigg| \notag \\
    &=& O(h_{X,M}). \notag 
    \EEA
    Plugging into the well known formula for the Christoffel symbols in terms of the metric tensor and leveraging estimates above, it follows immediately that 
    \BEA 
   \left|\Gamma^q_{st}\left( \alpha_{x_i} \right) -   \Gamma^q_{st}(\gamma_{x_i}) \right| = O(h_{X,M}). \notag 
    \EEA 
    This completes the proof. 
\end{proof}

\begin{proof}[Proof of Lemma \ref{local-integral-convergence-rate}]
    Note that,
    $$
    (\nabla W_N)_{\gamma_{x_i}} := g_{\gamma_{x_i}}^{sj} \Bigg(  \frac{\partial W_N^t}{\partial v_s} + W_N^p \Gamma^t_{ps}(\gamma_{x_i}) \Bigg) \frac{\partial}{ \partial v_t} \otimes \frac{\partial} {\partial v_j}.
    $$ 
    Similarly, 
    $$
    (\nabla \widehat{W}_N)_{\alpha_{x_i}} := g_{\alpha_{x_i}}^{sj} \Bigg(  \frac{\partial \widehat{W}_N^t}{\partial v_s} + \widehat{W}_N^p \Gamma^t_{ps}(\alpha_{x_i}) \Bigg) \frac{\partial}{ \partial v_t} \otimes \frac{\partial}{\partial v_j}.
    $$ 

 It follows from Lemmas \ref{local-metric-approximation-error} and \ref{local-vector-field-approximation-error} that 
\BEA 
|g_{\gamma_{x_i}}^{sj} - g_{\alpha_{x_i}}^{sj}| = O({\color{black}h_{X,M}^2}), & |(g_{\gamma_{x_i}})_{sj} - (g_{\alpha_{x_i}})_{sj})| = O({\color{black}h_{X,M}^2}) \notag \\
| W_N^p -  \widehat{W}_N^p| = O({\color{black}h_{X,M}^2}), & 
\left| \frac{\partial W_N^t}{\partial v_s} - \frac{\partial \widehat{W}_N^t}{\partial v_s} 
 \right| = O(h_{X,M}) \notag \\
 | \Gamma^t_{ps}(\gamma_{x_i}) - \Gamma^t_{ps}(\alpha_{x_i})| = O(h_{X,M}), & \left|\sqrt{ \textup{det}\left(g_{\gamma_{x_i}} \right)} - \sqrt{ \textup{det}\left(g_{\alpha_{x_i}} \right)}\right| = O({\color{black}h_{X,M}^2}). \notag 
\EEA
 Combining these estimates, simple algebraic manipulations show that 
\BEA 
 \Bigg | \Bigg\langle (\nabla U_N)_{\gamma_{x_i}} , (\nabla W_N)_{\gamma_{x_i}}
 \Bigg\rangle_{g_{\gamma_{x_i}}} \sqrt{ \textup{det}\left(g_{\gamma_{x_i}} \right)} -  \Bigg\langle  (\nabla \widehat{U}_N)_{\alpha_{x_i}}, (\nabla \widehat{W}_N)_{\alpha_{x_i}}
\Bigg\rangle_{g_{\alpha_{x_i}}} \sqrt{ \textup{det}\left(g_{\alpha_{x_i}} \right)} \Bigg| = O(h_{X,M}) \notag, 
\EEA
and 
\BEA 
 \Bigg | \int_{T_{i,r}} \Bigg[ \Bigg\langle (\nabla U_N)_{\gamma_{x_i}} , (\nabla W_N)_{\gamma_{x_i}}
 \Bigg\rangle_{g_{\gamma_{x_i}}} \sqrt{ \textup{det}\left(g_{\gamma_{x_i}} \right)} -  \Bigg\langle  (\nabla \widehat{U}_N)_{\alpha_{x_i}}, (\nabla \widehat{W}_N)_{\alpha_{x_i}}
\Bigg\rangle_{g_{\alpha_{x_i}}} \sqrt{ \textup{det}\left(g_{\alpha_{x_i}} \right)} \Bigg] dv_1 dv_2\Bigg| = O( \text{Vol}(T_{i,r}) h_{X,M}) \notag. 
\EEA
{\color{black}Since $T_{i,r} = O(h_{X,M}^2)$, this completes the proof for \eqref{lem4.6:eq1}. The bound for \eqref{lem4.6:eq2} 
can be obtained
following the same argument. In this case, the error rate is one-order higher than that of \eqref{lem4.6:eq1}
since it does not depend on first derivatives and the Christoffel symbols. } 
\end{proof}

\section*{Declaration of generative AI and AI-assisted technologies in the manuscript preparation process}

During the preparation of this work, the author(s) used ChatGPT to improve the clarity and grammar of the manuscript. The author(s) reviewed and edited the output as needed and take full responsibility for the content of the published article. 

\bibliographystyle{plain}  
\bibliography{references}

\end{document}